\documentclass[10pt]{article}
\usepackage{amsthm}
\usepackage{amsmath}
\usepackage{amssymb}
\usepackage{makeidx}

\theoremstyle{definition}
\newtheorem{definition}{Definition}[section]

\newtheorem{example}[definition]{Example}

\newtheorem{remark}[definition]{Remark}

\theoremstyle{plain}
\newtheorem{theorem}[definition]{Theorem}
\newtheorem{lemma}[definition]{Lemma}
\newtheorem{proposition}[definition]{Proposition}
\newtheorem{corollary}[definition]{Corollary}

\numberwithin{equation}{section}

\def\N{{\mathbb N}}

\begin{document}
\title{Taking Prime, Maximal and Two--class Congruences\\ Through Morphisms}
\author{Claudia MURE\c SAN\thanks{\small Dedicated to the memory of my dear grandmother, Floar\u a--Marioara Mure\c san}\\ \footnotesize University of Bucharest\\ \footnotesize Faculty of Mathematics and Computer Science\\ \footnotesize Academiei 14, RO 010014, Bucharest, Romania\\ \footnotesize E--mails: c.muresan@yahoo.com, cmuresan@fmi.unibuc.ro}
\date{\today }
\maketitle

\begin{abstract} In this paper we study prime, maximal and two--class congruences from the point of view of the relationships between them in various kinds of universal algebras, as well as their direct and inverse images through morphisms. This research has also produced a set of interesting results concerning the prime and the maximal congruences of several kinds of lattices.\end{abstract}

\section{Introduction}
\label{introduction}

In this paper, we study prime and maximal congruences in various kinds of algebras; we are interested in the cardinalities of the quotient algebras through these congruences and in the direct and inverse images of these congruences through morphisms. For the properties we obtain, we provide examples in lattices. We also prove a series of results concerning prime and maximal congruences in some classes of lattices.

The paper is structured in nine sections. In Section \ref{preliminaries}, we recall some previously known results from lattice theory, universal algebra and commutator theory; the results in the following sections are new, with the only exceptions of the results cited from other works and some of those in the final section, which we relate to the present context and derive from the other results we have obtained here; we also acknowledge that the characterizations for the primality of congruences which we have obtained in Section \ref{primired} are, up to a point, similar to the one from \cite{agl}.

Section \ref{primired} is concerned with some characterizations for prime congruences which serve us in the following sections for determining the prime congruences of the lattices in the different examples, and in the final section for some results on subdirectly irreducible algebras.

In Section \ref{admmorf}, we introduce two important types of morphisms that we study in the following sections: admissible and Max--admissible morphisms, defined by the property that the inverse images of prime, respectively maximal congruences through these morphisms are again prime, respectively maximal congruences. Then we provide some examples, which we also use in the sections which follow. The necessity for the study of admissible morphisms has appeared in the work for \cite{gulo}, and the related notion of Max--admissible morphisms naturally occurrs. In the following sections, we cite \cite{gulo} for several results concerning admissible morphisms.

In Section \ref{dirprodordsum}, we determine the prime, maximal and two--class congruences in direct products of algebras and finite ordinal sums of bounded lattices from the prime, maximal and two--class congruences of the terms of these direct products and ordinal sums, and prove that finite direct products and finite ordinal sums preserve the admissibility and Max--admissibility of morphisms.

In Section \ref{speclat}, we establish the relationships between the sets of the prime, maximal and two--class congruences in certain kinds of universal algebras and lattices, for the purpose of further studying admissible and Max--admissible morphisms based on these relationships.

In Section \ref{throughmorph}, we obtain several results on cardinalities of quotient algebras through congruences, for certain kinds of congruences, and in relation to the cardinalities of the quotient algebras through the direct and the inverse images of those congruences through morphisms. Then we use these results, as well as those from Section \ref{speclat}, to determine classes of algebras in which all morphisms are admissible and/or Max--admissible, as well as kinds of morphisms that are always admissible and/or Max--admissible, classified by the structures of their domain and their co--domain. 

In Section \ref{moreon(max)adm}, we prove other conditions which ensure the admissibility and/or Max--admissibility of morphisms, out of which we mention that surjectivity implies admissibility and Max--admissibility, but the converse does not hold. We also show that the study of admissibility and Max--admissibility reduces to embeddings, and prove that admissibility and Max--admissibility are preserved by quotients.

Section \ref{subdirirred} concludes the present paper, by some simple applications of the above to subdirect irreducibility of algebras; some of the results in this section are known; we just show how they can be derived from the previous results in this article.

\section{Preliminaries}
\label{preliminaries}

In this section, we recall some properties of equivalence relations, lattices, morphisms and congruences, and the commutator in congruence--modular varieties, which we need for making this paper self--contained. For a further study of the results on lattices that we point out here and those we shall recall in the following sections, we refer the reader to \cite{bal}, \cite{birkhoff}, \cite{blyth}, \cite{cwdw}, \cite{gratzer}, \cite{schmidt}; for the notions on universal algebras, we recommend \cite{bur}, \cite{gralgu}; for the results from commutator theory, see \cite{agl}, \cite{fremck}, \cite{koll}, \cite{owe}, \cite{urs}.

We shall denote by $\N $ the set of the natural numbers and by $\N ^*=\N \setminus \{0\}$. Let $M$ be a set. We shall denote by $|M|$ the cardinality of $M$, by ${\cal P}(M)$ the set of the subsets of $M$, by ${\rm Eq}(M)$ the set of the equivalences on $M$, by $\Delta _M=\{(x,x)\ |\ x\in M\}\in {\rm Eq}(M)$ and by $\nabla _M=M^2\in {\rm Eq}(M)$; for any $\theta \in {\rm Eq}(M)$, any $a\in M$ and any $S\subseteq M$, $a/\theta $ will denote the equivalence class of $a$ with respect to $\theta $, $S/\theta =\{x/\theta \ |\ x\in S\}$ and $p_{\theta }:M\rightarrow M/\theta $ shall be the canonical surjection. For any partition $\pi $ of $M$, we shall denote by $eq(\pi )$ the equivalence on $M$ which corresponds to $\pi $; thus we have $M/eq(\pi )=\pi $; if $\pi $ is finite, say $\pi =\{M_1,\ldots ,M_n\}$ for some $n\in \N ^*$, then we denote $eq(M_1,\ldots ,M_n)=eq(\pi )$. For any cardinal number $\kappa $, we shall denote by ${\rm Eq}_{\kappa }(M)=\{\theta \in {\rm Eq}(M)\ |\ |M/\theta |=\kappa \}$.

Let $I$ be a non--empty set, $(A_i)_{i\in I}$ and $(B_i)_{i\in I}$ be families of sets, $\displaystyle A=\prod _{i\in I}A_i$, $\displaystyle B=\prod _{i\in I}B_i$ and $f_i:A_i\rightarrow B_i$ for all $i\in I$. Then $\displaystyle f=\prod _{i\in I}f_i:A\rightarrow B$ shall have the usual componentwise definition. If $I=\overline{1,n}$ for some $n\in \N ^*$ and $f_1=\ldots =f_n=h$, then we denote $\displaystyle \prod _{i=1}^nf_i=h^n$. For any $S\subseteq A$, by $a=(a_i)_{i\in I}\in S$ we mean $a_i\in A_i$ for all $i\in I$, such that $a\in S$. If $R_i\subseteq A_i^2$ for all $i\in I$, then we denote by $\displaystyle \prod _{i\in I}R_i=\{((a_i)_{i\in I},(b_i)_{i\in I})\ |\ (\forall \, i\in I)\, ((a_i,b_i)\in R_i)\}$: direct product of binary relations. Clearly, if $\theta _i\in {\rm Eq}(A_i)$ for all $i\in I$, then $\displaystyle \prod _{i\in I}\theta _i\in {\rm Eq}(A)$.

Now let $M$ and $N$ sets, $h:M\rightarrow N$, $X\subseteq M^2$ and $Y\subseteq N^2$. We denote: $h(X)=h^2(X)=\{(h(a),h(b))\ |\ (a,b)\in X\}\subseteq N^2$ and $h^*(Y)=(h^2)^{-1}(Y)=\{(a,b)\in M^2\ |\ (h(a),h(b))\in Y\}\subseteq M^2$; with the direct images of these functions denoted in the usual way, it is straightforward that $h({\rm Eq}(M))\subseteq {\rm Eq}(h(M))$ and $h^*({\rm Eq}(N))\subseteq {\rm Eq}(M)$. We also denote by ${\rm Ker}(h)=\{(a,b)\in M^2\ |\ h(a)=h(b)\}=h^*(\Delta _N)\in {\rm Eq}(M)$: the {\em kernel} of $h$. It is immediate that $h(h^*(Y))=Y\cap h(M^2)=Y\cap h(\nabla _M)$, thus, if $h$ is surjective, then $h(h^*(Y))=Y$; therefore, if $h$ is surjective, then $h^*(Y)=X$ implies $h(X)=Y$, so $h^*$ is injective. If $h$ is injective, then $h^*(h(X))=X$, so $h^*$ is surjective. Clearly, if $M\subseteq N$ and $i:M\rightarrow N$ is the inclusion map, then ${\rm Ker}(i)=i^*(\Delta _N)=\Delta _M$. For any $\theta \in {\rm Eq}(M)$, we denote by $X/\theta =p_{\theta }(X)=\{(a/\theta ,b/\theta )\ |\ (a,b)\in X\}$.

With the notations above, let $U_i\subseteq A_i$ for all $i\in I$, $V_i\subseteq B_i$ for all $i\in I$, $\displaystyle U=\prod_{i\in I}U_i$, $\displaystyle V=\prod_{i\in I}V_i$, $R_i\subseteq A_i^2$ for all $i\in I$, $S_i\subseteq B_i^2$ for all $i\in I$, $\displaystyle R=\prod_{i\in I}R_i$ and $\displaystyle S=\prod_{i\in I}S_i$ as direct products of binary relations. Then, clearly, $\displaystyle f(U)=\prod_{i\in I}f_i(U_i)$ and $\displaystyle f^{-1}(V)=\prod_{i\in I}f_i^{-1}(V_i)$, hence $\displaystyle f(R)=\prod_{i\in I}f_i(R_i)$ and $\displaystyle f^*(S)=\prod_{i\in I}f_i^*(S_i)$ as direct products of binary relations, thus $\displaystyle {\rm Ker}(f)=f^*(\Delta _B)=f^*(\prod_{i\in I}\Delta _{B_i})=\prod_{i\in I}f_i^*(\Delta _{B_i})=\prod_{i\in I}{\rm Ker}(f_i)$.

Throughout this paper, whenever there is no danger of confusion, any algebra shall be designated by its support set. All algebras shall be considerred non--empty; by {\em trivial algebra} we mean one--element algebra, and by {\em non--trivial algebra} we mean algebra with at least two distinct elements. Any quotient algebra and any direct product of algebras shall be considerred with the operations defined canonically. Sometimes, for brevity, we shall denote by $A\cong B$ the fact that two algebras $A$ and $B$ of the same type are isomorphic.

Let $A$ be an algebra. We shall denote by ${\rm Con}(A)$ the set of the congruences of $A$ and, for any cardinality $\kappa $, by ${\rm Con}_{\kappa }(A)=\{\theta \in {\rm Con}(A)\ |\ |A/\theta |=\kappa \}={\rm Con}(A)\cap {\rm Eq}_{\kappa }(A)$. For each $X\subseteq A^2$, we shall denote by $Cg_A(X)$ the congruence of $A$ generated by $X$; for every $a,b\in A$, $Cg_A(\{(a,b)\})$ is also denoted by $Cg_A(a,b)$ and called the {\em principal congruence} of $A$ generated by $(a,b)$. Let $\phi \in {\rm Con}(A)$; $\phi $ is said to be {\em finitely generated} iff $\phi =Cg_A(X)$ for some finite subset $X$ of $A^2$; $\phi $ is called a {\em proper congruence} of $A$ iff $\phi \neq \nabla _A$. We recall that the {\em maximal congruences} of $A$ are the maximal elements of $({\rm Con}(A)\setminus \{\nabla _A\},\subseteq )$, and that the set of the maximal congruences of $A$ is denoted by ${\rm Max}(A)$. It is well known that $({\rm Con}(A),\vee ,\cup ,\Delta _A,\nabla _A)$ is a bounded lattice, orderred by set inclusion, where, for all $\phi ,\psi \in {\rm Con}(A)$, $\phi \vee \psi =Cg_A(\phi \cup \psi )$; moreover, ${\rm Con}(A)$ is a complete lattice, in which, for any family $(\phi 
_i)\subseteq {\rm Con}(A)$, $\displaystyle \bigvee _{i\in I}\phi _i=Cg_A(\bigcup _{i\in I}\phi _i)$. Obviously, $A$ is non--trivial iff $\Delta _A\neq \nabla _A$.

Throughout this paper, any (strict) order or lattice operation shall be denoted in the usual way, excepting particular cases such as lattices of congruences. Let $L$ be a lattice and $x\in L$. We recall that $x$ is called a {\em prime element} of $L$ iff, for all $a,b\in L$, $a\wedge b\leq x$ implies $a\leq x$ or $b\leq x$; $x$ is said to be {\em meet--irreducible} in $L$ iff, for all $a,b\in L$, $x=a\wedge b$ implies $x=a$ or $x=b$; $x$ is said to be {\em strictly meet--irreducible} iff there exists $\min \{y\in L\ |\ x<y\}$. Whenever $x$ has a unique successor in $L$, we shall denote that unique successor by $x^+$.

\begin{remark} Clearly:\begin{itemize}
\item $x$ is strictly meet--irreducible iff $x$ has a unique successor in $L$, namely $x^+=\min \{y\in L\ |\ x<y\}$;
\item if $x$ is strictly meet--irreducible, then $x$ is meet--irreducible, because, if $a,b\in L$ such that $x=a\wedge b$, so that $x\leq a$ and $x\leq b$, then $x=a$ or $x=b$, because otherwise we would have $x<a$ and $x<b$, thus $x^+\leq a$ and $x^+\leq b$, hence $x^+\leq a\wedge b=x<x^+$, a contradiction;
\item if $L$ has a $1$, then $\{y\in L\ |\ 1<y\}=\emptyset $, which has no minimum, thus $1$ is not strictly meet--irreducible; obviously, $1$ is meet--irreducible.\end{itemize}\label{strictmeetirred}\end{remark}

We shall denote by ${\rm Filt}(L)$, ${\rm Id}(L)$, ${\rm Max}_{\rm Filt}(L)$, ${\rm Max}_{\rm Id}(L)$, ${\rm Spec}_{\rm Filt}(L)$ and ${\rm Spec}_{\rm Id}(L)$ the sets of the filters, ideals, maximal filters, maximal ideals, prime filters and prime ideals of $L$, respectively. For any $X\subseteq L$, $[X)$, respectively $(X]$, shall be the filter, respectively the ideal of $L$ generated by $X$; for any $x\in L$, we shall denote by $[x)=[\{x\})$ and by $(x]=(\{x\}]$. The join in each of the lattices ${\rm Filt}(L)$ and ${\rm Id}(L)$ shall be denoted by $\vee $. If $L$ has a $1$, then $({\rm Filt}(L),\vee ,\cap ,\{1\},L)$ is a complete lattice, while, if $L$ has a $0$, then $({\rm Id}(L),\vee ,\cap ,\{0\},L)$ is a complete lattice. If $L$ is distributive, then $\varphi _L:{\rm Filt}(L)\rightarrow {\rm Con}(L)$ and $\chi _L:{\rm Id}(L)\rightarrow {\rm Con}(L)$ shall be the canonical lattice embeddings: for all $F\in {\rm Filt}(L)$ and all $I\in {\rm Id}(L)$, $\varphi _L(F)=\{(x,y)\in L^2\ |\ (\exists \, a\in F)\, (x\wedge a=y\wedge a)\}$ and $\chi _L(I)=\{(x,y)\in L^2\ |\ (\exists \, a\in I)\, (x\vee a=y\vee a)\}$; it is easy to prove that $F=1/\varphi _L(F)$ and $\varphi _L(F)$ is the smallest congruence of $L$ which has $F$ as a class; the dual goes for $\chi _L$. We recall that, if $L$ is a Boolean algebra, then its congruences coincide to those of its underlying lattice, and $\varphi _L$ and $\chi _L$ are bounded lattice isomorphisms. Note, also, that bounded lattice morphisms between Boolean algebras are Boolean morphisms. It is an immediate consequence of The Prime Filter Theorem that, in any distributive lattice, any proper filter equals the intersection of the prime filters that include it. The dual holds for ideals. We shall abbreviate by {\em ACC} the ascending chain condition for lattices. We shall denote by ${\cal D}$ the diamond, by ${\cal P}$ the pentagon and by ${\cal L}_n$ the $n$--element chain, for any $n\in \N ^*$.

\begin{remark} By \cite[Lemma $6$, p. $19$, and Lemma $7$, p. $20$]{gratzer}, any class of a congruence of a lattice $L$ is a convex sublattice of $L$, thus it has a unique writing as an intersection between a filter and an ideal of $L$. Clearly, if $S$ is a sublattice of the lattice $L$ and $\theta \in {\rm Con}(L)$, then $\theta \cap S^2\in {\rm Con}(S)$.\label{convex}\end{remark}

\begin{remark}{\rm \cite{bal}, \cite{blyth}, \cite{gratzer}, \cite{eudacs15}} Given any lattice $L$:\begin{itemize}
\item for any $\theta \in {\rm Con}(L)$, if $L$ has a $0$, then $0/\theta \in {\rm Id}(L)$, and, if $L$ has a $1$, then $1/\theta \in {\rm Filt}(L)$;
\item the mapping $P\mapsto L\setminus P$ is a bijection between ${\rm Spec}_{\rm Filt}(L)$ and ${\rm Spec}_{\rm Id}(L)$;

\item ${\rm Con}_2(L)=\{eq(P,L\setminus P)\ |\ P\in {\rm Spec}_{\rm Filt}(L)\}$.\end{itemize}\label{con2lat}\end{remark}

\begin{lemma}{\rm \cite{bal}, \cite{blyth}, \cite{gratzer}, \cite{eudacs15}} If $L$ is a chain, then:\begin{itemize}
\item the congruences of $L$ are exactly the equivalences on $L$ whose classes are convex;
\item any convex subset of $L$ is the class of a congruence of $L$;
\item ${\rm Spec}_{\rm Filt}(L)={\rm Filt}(L)\setminus \{\nabla _L\}$ and ${\rm Spec}_{\rm Id}(L)={\rm Id}(L)\setminus \{\nabla _L\}$;
\item for any $\theta \in {\rm Con}(L)$ and any $C,D\in L/\theta $, we have the following equivalences: $(\exists \, x\in C)\, (\exists \, y\in D)\, (x<y)$ iff $(\forall \, x\in C)\, (\forall \, y\in D)\, (x<y)$ iff $C<D$ in the chain $L/\theta $.\end{itemize}\label{chainconvex}\end{lemma}

An algebra $A$ is said to be {\em congruence--modular}, respectively {\em congruence--distributive}, iff the lattice ${\rm Con}(A)$ is modular, respectively distributive. An equational class ${\cal C}$ is said to be {\em congruence--modular}, respectively {\em congruence--distributive}, iff all algebras from ${\cal C}$ are congruence--modular, respectively congruence--distributive. The class of lattices is congruence--distributive; for instance, that of commutative rings is congruence--modular and it is not congruence--distributive.

Throughout the rest of this paper, ${\cal C}$ shall be an equational class of algebras of the same type, $A$ and $B$ shall be algebras from ${\cal C}$ and $f:A\rightarrow B$ shall be a morphism in ${\cal C}$.

Let us note that, if $I$ is a non--empty set, $(A_i)_{i\in I}$ and $(B_i)_{i\in I}$ are families of algebras in ${\cal C}$ and, for all $i\in I$, $f_i:A_i\rightarrow B_i$, then, clearly: $\displaystyle \prod _{i\in I}f_i$ is a morphism in ${\cal C}$ iff, for all $i\in I$, $f_i$ is a morphism in ${\cal C}$.

It is straightforward that, for any $\psi \in {\rm Con}(B)$, $f^*(\psi )\in {\rm Con}(A)$; thus ${\rm Ker}(f)\in {\rm Con}(A)$; and, for any $\phi \in {\rm Con}(A)$, $f(\phi )\in {\rm Con}(f(A))$; thus, if $f$ is surjective, then $f(\phi )\in {\rm Con}(B)$. It is well known that, for any $\theta \in {\rm Con}(A)$, $p_{\theta }$ is a surjective morphism and the mapping $\gamma \mapsto p_{\theta }(\gamma )=\gamma /\theta $ sets a bounded lattice isomorphism from $[\theta )$ to ${\rm Con}(A/\theta )$, so ${\rm Con}(A/\theta )=\{\gamma /\theta \ | \gamma \in [\theta )\}$ and, for all $\gamma \in [\theta )$, $p_{\theta }^*(p_{\theta }(\gamma ))=p_{\theta }^*(\gamma /\theta )=\gamma $, thus ${\rm Ker}(p_{\theta })=p_{\theta }^*(\Delta _{A/\theta })=p_{\theta }^*(\theta /\theta )=\theta $. Thus, for any $\gamma \in [\theta )$ and any $a,b\in A$, the following hold: $(a/\theta ,b/\theta )\in \gamma /\theta $ iff $(p_{\theta }(a),p_{\theta }(b))\in p_{\theta }(\gamma )$ iff $(a,b)\in p_{\theta }^*(p_{\theta }(\gamma ))$ iff $(a,b)\in \gamma $. Hence, for any $\alpha ,\beta \in [\theta )$: $\alpha /\theta =\beta /\theta $ iff $\alpha =\beta $, and: $\alpha /\theta \subseteq \beta /\theta $ iff $\alpha \subseteq \beta $.

\begin{remark} By the above, for any $\theta \in {\rm Con}(A)$, ${\rm Con}(A/\theta )=\{\psi /\theta \ |\ \psi \in [\theta )\}$, hence: $\theta \in {\rm Max}(A)$ iff $\theta \neq \nabla _A$ and $[\theta )=\{\theta ,\nabla _A\}$ iff $[\theta )\cong {\cal L}_2$ iff ${\rm Con}(A/\theta )=\{\Delta _{A/\theta },\nabla _{A/\theta }\}$ and $\Delta _{A/\theta }\neq \nabla _{A/\theta }$ iff ${\rm Con}(A/\theta )\cong {\cal L}_2$.\label{congrcat}\end{remark}

\begin{theorem}{\rm \cite{fremck}} If ${\cal C}$ is congruence--modular, then, for each member $M$ of ${\cal C}$, there exists a unique binary operation $[\cdot ,\cdot ]_M$ on ${\rm Con}(M)$, called the {\em commutator of $M$}, such that, for all $\alpha ,\beta \in {\rm Con}(M)$, $[\alpha ,\beta ]_M=\min \{\mu \in {\rm Con}(M)\ |\ \mu \subseteq \alpha \cap \beta $ and, for any algebra $N$ from ${\cal C}$ and any surjective morphism $h:M\rightarrow N$, $\mu \vee {\rm Ker}(h)=h^*([h(\alpha \vee {\rm Ker}(h)),h(\beta \vee {\rm Ker}(h))]_N)$.\label{wow}\end{theorem}

\begin{proposition}{\rm \cite{fremck}} If ${\cal C}$ is congruence--modular, then the commutator in ${\cal C}$ is:\begin{itemize}
\item included in the intersection: $[\alpha ,\beta ]_A\subseteq \alpha \cap \beta $ for all $\alpha ,\beta \in {\rm Con}(A)$;
\item commutative, that is $[\alpha ,\beta ]_A=[\beta ,\alpha ]_A$ for all $\alpha ,\beta \in {\rm Con}(A)$;
\item increasing in both arguments, that is, for all $\alpha ,\beta ,\phi ,\psi \in {\rm Con}(A)$, if $\alpha \subseteq \beta $ and $\phi \subseteq \psi $, then $[\alpha ,\phi ]_A\subseteq [\beta ,\psi ]_A$;
\item distributive in both arguments with respect to arbitrary joins, that is, for any non--empty families $(\alpha _i)_{i\in I}$ and $(\beta _j)_{j\in J}$ of congruences of $A$, we have $\displaystyle [\bigvee _{i\in I}\alpha _i,\bigvee _{j\in J}\beta _j]_A=\bigvee _{i\in I}\bigvee _{j\in J}[\alpha _i,\beta _j]_A$.\end{itemize}\label{1.3}\end{proposition}

\begin{theorem}{\rm \cite{fremck}} If ${\cal C}$ is congruence--distributive, then, in each member of ${\cal C}$, the commutator coincides to the intersection of congruences.\label{distrib}\end{theorem}

Following \cite{fremck}, if ${\cal C}$ is congruence--modular and $\phi $ is a proper congruence of $A$, then we call $\phi $ a {\em prime congruence} iff, for all $\alpha ,\beta \in {\rm Con}(A)$, $[\alpha ,\beta ]_A\subseteq \phi $ implies $\alpha \subseteq \phi $ or $\beta \subseteq \phi $. The set of the prime congruences of $A$ shall be denoted by ${\rm Spec}(A)$. Note that not every algebra in a congruence--modular equational class has prime congruences.

\begin{remark} Theorem \ref{distrib} shows that, if ${\cal C}$ is congruence--distributive, then the prime congruences of $A$ are exactly the prime elements of the lattice ${\rm Con}(A)$. Note that the same holds if ${\cal C}$ is congruence--modular and the commutator in $A$ equals the intersection of congruences.\label{primdistrib}\end{remark}

We recall that ${\cal C}$ is said to be {\em semi--degenerate} iff no non--trivial algebra in ${\cal C}$ has trivial subalgebras. For instance, the class of bounded lattices is semi--degenerate, and so is any class of bounded orderred structures.

\begin{proposition}{\rm \cite{koll}} The following are equivalent:\begin{itemize}
\item ${\cal C}$ is semi--degenerate;
\item for all members $M$ of ${\cal C}$, $\nabla _M$ is finitely generated.\end{itemize}\label{2.6}\end{proposition}

\begin{lemma}{\rm \cite[Theorem $5.3$]{agl}} If ${\cal C}$ is congruence--modular and $\nabla _A$ is finitely generated, then:\begin{itemize}
\item any proper congruence of $A$ is included in a maximal congruence of $A$;
\item any maximal congruence of $A$ is prime.\end{itemize}\label{folclor}\end{lemma}

\begin{remark}\begin{itemize}\item By Lemma \ref{folclor}, if ${\cal C}$ is congruence--modular, $\nabla _A$ is finitely generated and $A$ is non--trivial, so that $\Delta _A$ is a proper congruence of $A$, then $\emptyset \neq {\rm Max}(A)\subseteq {\rm Spec}(A)$.\item Proposition \ref{2.6} shows that, if ${\cal C}$ is congruence--modular and semi--degenerate, then every member of ${\cal C}$ fulfills the properties stated in Lemma \ref{folclor}.\end{itemize}\label{rfolclor}\end{remark}

\begin{proposition}{\rm \cite[Theorem 8.5, p. 85]{fremck}} If ${\cal C}$ is congruence--modular, then the following are equivalent:\begin{itemize}
\item for any algebra $M$ from ${\cal C}$, $[\nabla _M,\nabla _M]_M=\nabla _M$;
\item for any algebra $M$ from ${\cal C}$ and any $\theta \in {\rm Con}(M)$, $[\theta ,\nabla _M]_M=\theta $;
\item for any $n\in \N ^*$ and any algebras $M_1,\ldots ,M_n$ from ${\cal C}$, $\displaystyle {\rm Con}(\prod _{i=1}^nM_i)=\{\prod _{i=1}^n\theta _i\ |\ (\forall \, i\in \overline{1,n})\, (\theta _i\in {\rm Con}(M_i))$.\end{itemize}\label{prodcongr}\end{proposition}

\begin{remark} Clearly, by Theorem \ref{distrib}, if ${\cal C}$ is congruence--distributive, then ${\cal C}$ fulfills the equivalent conditions from Proposition \ref{prodcongr}.\label{distribprod}\end{remark}

\begin{lemma}{\rm \cite[Lemma 5.2]{agl}} If ${\cal C}$ is congruence--modular and semi--degenerate, then ${\cal C}$ fulfills the equivalent conditions from Proposition \ref{prodcongr}.\label{semidprod}\end{lemma}

\begin{proposition}{\rm \cite[Theorem 5.17, p. 48]{owe}} Assume that ${\cal C}$ is congruence--modular, and let $n\in \N ^*$, $M_1,\ldots ,M_n$ be algebras from ${\cal C}$, $\displaystyle M=\prod _{i=1}^nM_i$ and, for all $i\in \overline{1,n}$, $\alpha _i,\beta _i\in {\rm Con}(M_i)$. Then: $\displaystyle [\prod _{i=1}^n\alpha _i,\prod _{i=1}^n\beta _i]_M=\prod _{i=1}^n[\alpha _i,\beta _i]_{M_i}$.\label{comutprod}\end{proposition}

\section{Primality Versus Meet--irreducibility of Congruences}
\label{primired}

In this section, we present some characterizations for prime congruences that will be useful in the examples we shall provide in the following sections. Throughout this section, $L$ shall be a lattice and $x\in L$.

\begin{lemma} Then the following are equivalent:\begin{enumerate}
\item\label{genmeetirred1} $x$ is strictly meet--irreducible in $L$;
\item\label{genmeetirred2} $x$ is meet--irreducible in $L$ and $x$ has successors in $L$.\end{enumerate}\label{genmeetirred}\end{lemma}

\begin{proof} (\ref{genmeetirred1})$\Rightarrow $(\ref{genmeetirred2}): By Remark \ref{strictmeetirred}.

\noindent (\ref{genmeetirred2})$\Rightarrow $(\ref{genmeetirred1}): If $a$ and $b$ would be two distinct successors of $x$ in $L$, then we would have $x=a\wedge b$, $x<a$ and $x<b$, which would contradict the fact that $x$ is meet--irreducible. Thus $x$ has a unique successor in $L$, which means that $x$ is strictly meet--irreducible by Remark \ref{strictmeetirred}.\end{proof}

\begin{proposition} If $[x)$ is finite, then the following are equivalent:\begin{enumerate}
\item\label{meetirred1} $x$ is strictly meet--irreducible in $L$;
\item\label{meetirred2} $x$ is meet--irreducible in $L$ and $x\neq 1$.\end{enumerate}\label{meetirred}\end{proposition}

\begin{proof} Clearly, if $L$ has finite filters, then $L$ has a $1$. Now apply Lemma \ref{genmeetirred} and the fact that, if $[x)$ is finite, then $1$ is the only element of $[x)$ without successors in $L$.\end{proof}

\begin{corollary} If the lattice $L$ is finite, then the following are equivalent:\begin{enumerate}
\item\label{cormeetirred1} $x$ is strictly meet--irreducible in $L$;
\item\label{cormeetirred2} $x$ is meet--irreducible in $L$ and $x\neq 1$.\end{enumerate}\label{cormeetirred}\end{corollary}

The following characterization for the primality of congruences is, up to a point, similar to the one from \cite[Proposition 1.2]{agl}, so we may say that this is simply a tinting of this result of P. Agliano:

\begin{proposition} Assume that ${\cal C}$ is congruence--modular, and let $\phi \in {\rm Con}(A)$ such that $[\phi ,\phi ]=\phi $. Then:\begin{enumerate}
\item\label{erhardgen1} if $\phi $ is strictly meet--irreducible and $[\phi ^+,\phi ^+]_A\neq \phi $, then $\phi \in {\rm Spec}(A)$;
\item\label{erhardgen2} if $\phi \in {\rm Spec}(A)$, then $\phi $ is proper and meet--irreducible and $[\alpha ,\beta ]_A\neq \phi $ for any $\alpha ,\beta \in {\rm Con}(A)$ such that $\phi \subsetneq \alpha $ and $\phi \subsetneq \beta $;
\item\label{erhardgen3} if $[\phi )$ is finite, then: $\phi \in {\rm Spec}(A)$ iff $\phi $ is strictly meet--irreducible and $[\phi ^+,\phi ^+]_A\neq \phi $ iff $\phi $ is proper and meet--irreducible and $[\alpha ,\beta ]_A\neq \phi $ for any $\alpha ,\beta \in {\rm Con}(A)$ such that $\phi \subsetneq \alpha $ and $\phi \subsetneq \beta $;
\item\label{erhardgen4} if ${\rm Con}(A)$ is finite, then: $\phi \in {\rm Spec}(A)$ iff $\phi $ is strictly meet--irreducible and $[\phi ^+,\phi ^+]_A\neq \phi $ iff $\phi $ is proper and meet--irreducible and $[\alpha ,\beta ]_A\neq \phi $ for any $\alpha ,\beta \in {\rm Con}(A)$ such that $\phi \subsetneq \alpha $ and $\phi \subsetneq \beta $.\end{enumerate}\label{erhardgen}\end{proposition}

\begin{proof} (\ref{erhardgen1}) Assume that $\phi $ is strictly meet--irreducible and $[\phi ^+,\phi ^+]_A\neq \phi $, and let $\alpha ,\beta \in {\rm Con}(A)$ such that $[\alpha ,\beta ]_A\subseteq \phi $. Then, by Proposition \ref{1.3}, $[\alpha \vee \phi ,\beta \vee \phi ]_A=[\alpha ,\beta ]_A\vee [\alpha ,\phi ]_A\vee [\phi ,\beta ]_A\vee [\phi ,\phi ]_A\subseteq \phi \vee (\alpha \cap \phi )\vee (\phi \vee \beta )\vee \phi \subseteq \phi \vee \phi \vee \phi \vee \phi =\phi $, and we have $\phi \subseteq \alpha \vee \phi $ and $\phi \subseteq \beta \vee \phi $. Assume by absurdum that $\phi \subsetneq \alpha \vee \phi $ and $\phi \subsetneq \beta \vee \phi $, so that $\phi ^+ \subseteq \alpha \vee \phi $ and $\phi ^+\subseteq \beta \vee \phi $. Then, by Proposition \ref{1.3}, $\phi =[\phi ,\phi ]_A\subseteq [\phi ^+,\phi ^+]_A\subseteq [\alpha \vee \phi ,\beta \vee \phi ]_A\subseteq \phi $, hence $[\phi ^+,\phi ^+]_A=\phi $; we have a contradiction. Thus $\alpha \vee \phi =\phi $ or $\beta \vee \phi =\phi $, that is $\alpha \subseteq \phi $ or $\beta \subseteq \phi $. Hence $\phi \in {\rm Spec}(A)$.

\noindent (\ref{erhardgen2}) Assume that $\phi \in {\rm Spec}(A)$, so $\phi \neq \nabla _A$. Let $\alpha ,\beta \in {\rm Con}(A)$ such that $\alpha \cap \beta =\phi $. Then $\phi \subseteq \alpha $, $\phi \subseteq \beta $ and $[\alpha ,\beta ]_A\subseteq \alpha \cap \beta =\phi $. Since $\phi \in {\rm Spec}(A)$, it follows that $\alpha \subseteq \phi $ or $\beta \subseteq \phi $, thus $\alpha =\phi $ or $\beta =\phi $. Hence $\phi $ is meet--irreducible. Now let $\alpha ,\beta \in {\rm Con}(A)$ such that $\phi \subsetneq \alpha $ and $\phi \subsetneq \beta $, and assume by absurdum that $[\alpha ,\beta ]_A=\phi \subseteq \phi $. Since $\phi \in {\rm Spec}(A)$, we have $\alpha \subseteq \phi $ or $\beta \subseteq \phi $, thus $\alpha =\phi $ or $\beta =\phi $, and $\alpha \neq \phi $ and $\beta \neq \phi $, a contradiction. Thus $[\alpha ,\beta ]_A\neq \phi $.

\noindent (\ref{erhardgen3}) By (\ref{erhardgen1}), (\ref{erhardgen2}), Propositions \ref{meetirred} and \ref{1.3} and the fact that, when $\phi ^+$ exists, we have $\phi \subsetneq \phi ^+$ and, for any $\alpha \in {\rm Con}(A)$, $\phi \subsetneq \alpha $ iff $\phi ^+\subseteq \alpha $.

\noindent (\ref{erhardgen4}) By (\ref{erhardgen3}).\end{proof}

\begin{corollary} Assume that ${\cal C}$ is congruence--distributive, or that it is is congruence--modular and the commutator in $A$ equals the intersection, and let $\phi \in {\rm Con}(A)$. Then:\begin{enumerate}
\item\label{mypartic1} if $\phi $ is strictly meet--irreducible, then $\phi \in {\rm Spec}(A)$;
\item\label{mypartic2} if $\phi \in {\rm Spec}(A)$, then $\phi $ is proper and meet--irreducible; 
\item\label{mypartic3} if $[\phi )$ is finite, then: $\phi \in {\rm Spec}(A)$ iff $\phi $ is strictly meet--irreducible iff $\phi $ is proper and meet--irreducible;
\item\label{mypartic4} if ${\rm Con}(A)$ is finite, then: $\phi \in {\rm Spec}(A)$ iff $\phi $ is strictly meet--irreducible iff $\phi $ is proper and meet--irreducible.\end{enumerate}\label{mypartic}\end{corollary}

\begin{proof} By Proposition \ref{erhardgen} and Theorem \ref{distrib}, which ensures us that $[\phi ,\phi ]_A=\phi \cap \phi =\phi $, that, when $\phi ^+$ exists, $[\phi ^+,\phi ^+]_A=\phi ^+\cap \phi ^+=\phi ^+\neq \phi $, and, for any $\alpha ,\beta \in {\rm Con}(A)$, $[\alpha ,\beta ]_A=\phi $ means that $\alpha \cap \beta =\phi $.

Note that (\ref{mypartic2}) is also a direct consequence of \cite[Proposition 1.2]{agl}.\end{proof}

\begin{proposition} Assume that ${\cal C}$ is congruence--distributive, or that it is is congruence--modular and the commutator in $A$ equals the intersection. If ${\rm Con}(A)$ is a Boolean algebra, then ${\rm Spec}(A)={\rm Max}(A)$.\label{conbool}\end{proposition}

\begin{proof} It is well known that the prime ideals of a Boolean algebra coincide to its maximal ideals. Let $\phi \in {\rm Con}(A)$. Then, by Remark \ref{primdistrib}, the definition of a prime element and that of a prime ideal, we have the following: $\phi \in {\rm Spec}(A)$ iff, for all $\alpha ,\beta \in {\rm Con}(A)$, $\alpha \cap \beta \subseteq \phi $ iff $\alpha \subseteq \phi $ or $\beta \subseteq \phi $, iff, for all $\alpha ,\beta \in {\rm Con}(A)$, $\alpha \cap \beta \in (\phi ]$ iff $\alpha \in (\phi ]$ or $\beta \in (\phi ]$, iff $(\phi ]$ is a prime ideal of ${\rm Con}(A)$ iff $(\phi ]$ is a maximal ideal of ${\rm Con}(A)$ iff $\phi $ is a co--atom of the Boolean algebra ${\rm Con}(A)$ iff $\phi \in {\rm Max}(A)$. Therefore ${\rm Spec}(A)={\rm Max}(A)$.\end{proof}

In the examples that follow, we shall use Remark \ref{convex} to determine the congruences of the lattices, and Corollary \ref{mypartic}, (\ref{mypartic4}), and Remark \ref{strictmeetirred}, to determine their prime congruences, which, since their lattices of congruences are finite, are exactly the elements of these lattices which have unique successors in these lattices. The configurations of their lattices of congruences will give us their maximal congruences.

\section{Admissible and Max--admissible Morphisms}
\label{admmorf}

In \cite{gulo}, we study properties Going Up and Lying Over in Congruence--modular Algebras. The study of these properties in this general context necessitates a preliminary study of a certain kind of morphisms we have called {\em admissible morphisms}. Here we just recall their definition, and we also define another kind of admissibility for morphisms, then we give some examples. We shall continue the study of these kinds of morphisms in the following sections.

Following \cite{gulo}, if ${\cal C}$ is congruence--modular, then we call $f$ an {\em admissible morphism} iff $f^*(\psi )\in {\rm Spec}(A)$ for all $\psi \in {\rm Spec}(B)$. By analogy, we call $f$ a {\em Max--admissible morphism} iff $f^*(\psi )\in {\rm Max}(A)$ for all $\psi \in {\rm Max}(B)$. These two notions are non--trivial and independent of each other, as shown by the following example:

\begin{example} Let ${\cal L}_2^2$, ${\cal P}$ and ${\cal D}$ have the elements denoted as below, $i:{\cal L}_2^2\rightarrow {\cal P}$ and $j:{\cal L}_2^2\rightarrow {\cal D}$ be the canonical bounded lattice embeddings and $g: {\cal P}\rightarrow {\cal L}_2^2$ and $h:{\cal P}\rightarrow {\cal D}$ be the bounded lattice morphisms given by the following tables:\vspace*{-20pt}

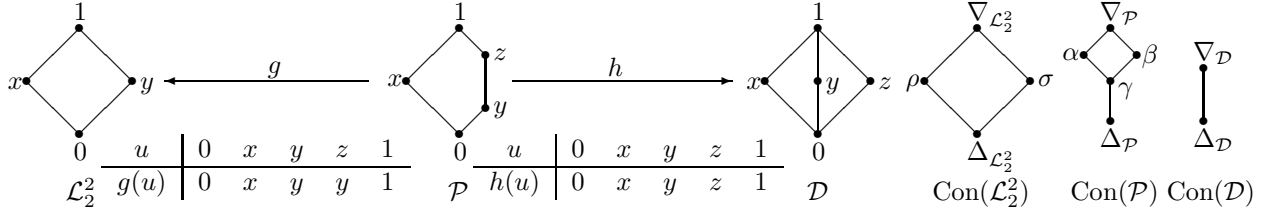
\begin{center}\begin{tabular}{cccccccc}
\begin{picture}(40,80)(0,0)
\put(30,25){\line(1,1){20}}
\put(30,25){\line(-1,1){20}}
\put(30,65){\line(1,-1){20}}
\put(30,65){\line(-1,-1){20}}
\put(30,25){\circle*{3}}
\put(10,45){\circle*{3}}
\put(3,42){$x$}
\put(50,45){\circle*{3}}
\put(53,42){$y$}
\put(30,65){\circle*{3}}
\put(28,15){$0$}
\put(28,68){$1$}
\put(25,0){${\cal L}_2^2$}

\end{picture}
&\hspace*{-17pt}
\begin{picture}(100,80)(0,0)
\put(101,45){\vector(-1,0){77}}
\put(63,48){$g$}
\put(0,10){\begin{tabular}{c|ccccc}
$u$ & $0$ & $x$ & $y$ & $z$ & $1$\\ \hline 
$g(u)$ & $0$ & $x$ & $y$ & $y$ & $1$\end{tabular}}
\end{picture}
&\hspace*{-10pt}
\begin{picture}(40,80)(0,0)
\put(30,25){\line(1,1){10}}
\put(30,25){\line(-1,1){20}}
\put(30,65){\line(1,-1){10}}
\put(30,65){\line(-1,-1){20}}
\put(40,35){\line(0,1){20}}
\put(30,25){\circle*{3}}
\put(10,45){\circle*{3}}
\put(3,43){$x$}
\put(40,35){\circle*{3}}
\put(43,33){$y$}
\put(30,65){\circle*{3}}
\put(40,55){\circle*{3}}
\put(43,53){$z$}
\put(26,0){${\cal P}$}
\put(28,15){$0$}
\put(28,68){$1$}
\end{picture}
&\hspace*{-20pt}
\begin{picture}(100,80)(0,0)
\put(15,45){\vector(1,0){83}}

\put(51,47){$h$}
\put(0,10){\begin{tabular}{c|ccccc}
$u$ & $0$ & $x$ & $y$ & $z$ & $1$\\ \hline 
$h(u)$ & $0$ & $x$ & $y$ & $z$ & $1$\end{tabular}}
\end{picture}
&\hspace*{-15pt}
\begin{picture}(40,80)(0,0)
\put(30,25){\line(0,1){40}}
\put(30,25){\line(1,1){20}}
\put(30,25){\line(-1,1){20}}
\put(30,65){\line(1,-1){20}}
\put(30,65){\line(-1,-1){20}}
\put(30,25){\circle*{3}}
\put(10,45){\circle*{3}}
\put(3,42){$x$}
\put(30,45){\circle*{3}}
\put(33,42){$y$}
\put(50,45){\circle*{3}}
\put(53,42){$z$}
\put(30,65){\circle*{3}}
\put(28,15){$0$}
\put(28,68){$1$}
\put(26,0){${\cal D}$}
\end{picture}
&\hspace*{5pt}
\begin{picture}(40,80)(0,0)
\put(30,25){\line(1,1){20}}
\put(30,25){\line(-1,1){20}}
\put(30,65){\line(1,-1){20}}
\put(30,65){\line(-1,-1){20}}
\put(30,25){\circle*{3}}
\put(10,45){\circle*{3}}
\put(3,43){$\rho $}
\put(50,45){\circle*{3}}
\put(53,43){$\sigma $}
\put(30,65){\circle*{3}}
\put(26,15){$\Delta _{{\cal L}_2^2}$}
\put(26,68){$\nabla _{{\cal L}_2^2}$}
\put(13,0){${\rm Con}({\cal L}_2^2)$}
\end{picture}
&\hspace*{5pt}
\begin{picture}(40,80)(0,0)
\put(20,45){\line(1,1){10}}
\put(20,45){\line(-1,1){10}}
\put(20,65){\line(1,-1){10}}
\put(20,65){\line(-1,-1){10}}
\put(20,45){\line(0,-1){15}}
\put(16,20){$\Delta _{\cal P}$}
\put(16,68){$\nabla _{\cal P}$}
\put(20,30){\circle*{3}}
\put(23,40){$\gamma $}
\put(2,53){$\alpha $}
\put(32,52){$\beta $}
\put(20,45){\circle*{3}}
\put(10,55){\circle*{3}}
\put(30,55){\circle*{3}}
\put(20,65){\circle*{3}}
\put(5,0){${\rm Con}({\cal P})$}
\end{picture}
&\hspace*{-20pt}
\begin{picture}(40,80)(0,0)
\put(20,30){\line(0,1){20}}
\put(20,30){\circle*{3}}
\put(20,50){\circle*{3}}
\put(16,20){$\Delta _{\cal D}$}
\put(16,53){$\nabla _{\cal D}$}
\put(6,0){${\rm Con}({\cal D})$}\end{picture}\end{tabular}\end{center}

${\rm Con}({\cal L}_2^2)=\{\Delta _{{\cal L}_2^2},\rho ,\sigma ,\nabla _{{\cal L}_2^2}\}\cong {\cal L}_2^2$, where $\rho =eq(\{0,x\},\{y,1\})$ and $\sigma =eq(\{0,y\},\{x,1\})$, so ${\rm Spec}({\cal L}_2^2)={\rm Max}({\cal L}_2^2)=\{\rho ,\sigma \}$. ${\rm Con}({\cal D})=\{\Delta _{\cal D},\nabla _{\cal D}\}\cong {\cal L}_2$, so ${\rm Spec}({\cal D})={\rm Max}({\cal D})=\{\Delta _{\cal D}\}$. See above the lattice of congruences of ${\cal P}$, where $\alpha =eq(\{0,y,z\},\{x,1\})$, $\beta =eq(\{0,x\},\{y,z,1\})$ and $\gamma =eq(\{0\},\{x\},\{y,z\},\{1\})$, and notice that ${\rm Spec}({\cal P})=\{\Delta _{\cal P},\alpha ,\beta \}$ and ${\rm Max}({\cal P})=\{\alpha ,\beta \}$.

$g^*(\rho )=\beta \in {\rm Max}({\cal P})\subset {\rm Spec}({\cal P})$ and $g^*(\sigma )=\alpha \in {\rm Max}({\cal P})\subset {\rm Spec}({\cal P})$, so $g$ is both admissible and Max--admissible.

$j^*(\Delta _{\cal D})=\Delta _{{\cal L}_2^2}\notin {\rm Spec}({\cal L}_2^2)={\rm Max}({\cal L}_2^2)$, thus $j$ is neither admissible, nor Max--admissible.

$i^*(\Delta _{\cal P})=\Delta _{{\cal L}_2^2}\notin {\rm Spec}({\cal L}_2^2)$, $i^*(\alpha )=\sigma \in {\rm Max}({\cal L}_2^2)$ and $i^*(\beta )=\rho \in {\rm Max}({\cal L}_2^2)$, so $i$ is Max--admissible and it is not admissible. 

$h^*(\Delta _{\cal D})=\Delta _{\cal P}\in {\rm Spec}({\cal P})\setminus {\rm Max}({\cal P})$, thus $h$ is admissible and it is not Max--admissible.

Because it will prove important later on, here is an example of a morphism which is both admissible and Max--admissible, but it is not surjective and does not have the co--domain given by a bounded distributive lattice or a lattice which can be obtained through the constructions in Proposition \ref{spec2cls} below: let $k:{\cal D}\rightarrow E$ be the canonical bounded lattice embedding of ${\cal D}$ into the lattice $E$ given by the following Hasse diagram, embedding which is clearly not surjective:\vspace*{-40pt}

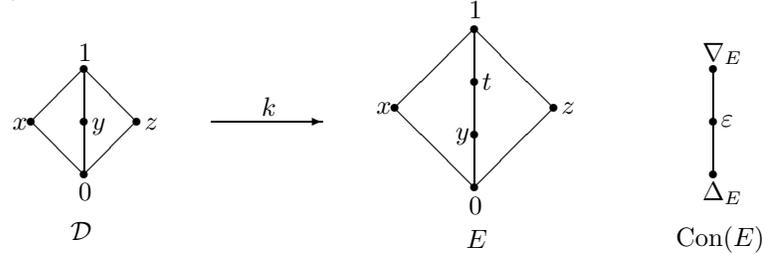
\begin{center}\begin{tabular}{cccc}
\begin{picture}(40,80)(0,0)
\put(30,25){\line(0,1){40}}
\put(30,25){\line(1,1){20}}
\put(30,25){\line(-1,1){20}}
\put(30,65){\line(1,-1){20}}
\put(30,65){\line(-1,-1){20}}
\put(30,25){\circle*{3}}
\put(10,45){\circle*{3}}
\put(3,42){$x$}
\put(30,45){\circle*{3}}
\put(33,42){$y$}
\put(50,45){\circle*{3}}
\put(53,42){$z$}
\put(30,65){\circle*{3}}
\put(28,15){$0$}
\put(28,68){$1$}
\put(25,0){${\cal D}$}
\end{picture}
&\hspace*{10pt}
\begin{picture}(47,80)(0,0)
\put(13,45){\vector(1,0){42}}
\put(32,47){$k$}
\end{picture}
&\hspace*{10pt}
\begin{picture}(40,120)(0,0)
\put(40,20){\line(0,1){60}}
\put(40,20){\line(1,1){30}}
\put(40,20){\line(-1,1){30}}
\put(40,80){\line(1,-1){30}}
\put(40,80){\line(-1,-1){30}}
\put(40,20){\circle*{3}}
\put(40,40){\circle*{3}}
\put(40,60){\circle*{3}}
\put(40,80){\circle*{3}}
\put(10,50){\circle*{3}}
\put(70,50){\circle*{3}}
\put(3,48){$x$}
\put(33,38){$y$}
\put(43,57){$t$}
\put(73,48){$z$}
\put(38,10){$0$}
\put(38,84){$1$}
\put(37,-3){$E$}
\end{picture}
&\hspace*{55pt}
\begin{picture}(40,100)(0,0)
\put(20,25){\line(0,1){40}}
\put(20,25){\circle*{3}}
\put(20,45){\circle*{3}}
\put(20,65){\circle*{3}}
\put(16,15){$\Delta _E$}
\put(23,43){$\varepsilon $}
\put(16,68){$\nabla _E$}
\put(6,-2){${\rm Con}(E)$}
\end{picture}\end{tabular}\end{center}

${\rm Con}(E)=\{\Delta _E,\varepsilon ,\nabla _E\}\cong {\cal L}_3$, where $\varepsilon =eq(\{0\},\{x\},\{y,t\},\{z\},\{1\})$, so ${\rm Spec}(E)=\{\Delta _E,\varepsilon \}$ and ${\rm Max}(E)=\{\varepsilon \}$. $k^*(\Delta _E)=k^*(\varepsilon )=\Delta _{\cal D}\in {\rm Spec}({\cal D})={\rm Max}({\cal D})$, so $k$ is admissible and Max--admissible.\label{fiectip}\end{example}

\begin{remark} Because we shall use this later, let us also note that ${\rm Con}({\cal L}_2)=\{\Delta _{{\cal L}_2},\nabla _{{\cal L}_2}\}\cong {\cal L}_2$. Just as ${\cal L}_2^2$, this is a finite Boolean lattice, thus it is isomorphic to its lattice of congruences.\label{lant2}\end{remark}

\section{Congruences in Direct Products of Algebras and Ordinal Sums of Bounded Lattices}
\label{dirprodordsum}

The first results in this section refer to direct products of algebras with the property that all their congruences are products of congruences of the terms of those direct products; for such direct products, we determine the form of the prime, maximal and two--class congruences; in the following sections, it will become clear why these kinds of congruences are important and related to each other. Then we do the same for finite ordinal sums of bounded lattices, and we prove that admissibility and Max--admissibility are preserved by finite direct products and, in the case of lattices, also by finite ordinal sums. 

\begin{remark}\begin{itemize}
\item If $(A_i)_{i\in I}$ is a non--empty family of sets and $\alpha _i\in {\rm Eq}(A_i)$ for all $i\in I$, then, as stated in Section \ref{preliminaries}, $\displaystyle \prod _{i\in I}\alpha _i\in {\rm Eq}(\prod _{i\in I}A_i)$, and it is immediate that $\displaystyle |(\prod _{i\in I}A_i)/(\prod _{i\in I}\alpha _i)|=|\prod _{i\in I}A_i/\alpha _i|=\prod _{i\in I}|A_i/\alpha _i|$, because the map $\displaystyle (a_i/\alpha _i)_{i\in I}\mapsto (a_i)_{i\in I}/(\prod _{i\in I}\alpha _i)$ sets a bijection from $\displaystyle \prod _{i\in I}A_i/\alpha _i$ to $\displaystyle (\prod _{i\in I}A_i)/(\prod _{i\in I}\alpha _i)$.
\item Furthermore, if $(A_i)_{i\in I}$ is a family of algebras from ${\cal C}$ and $\alpha _i\in {\rm Con}(A_i)$ for all $i\in I$, then, clearly, $\displaystyle \prod _{i\in I}\alpha _i\in {\rm Con}(\prod _{i\in I}A_i)$ and the map defined above is an isomorphism between the algebras $\displaystyle \prod _{i\in I}A_i/\alpha _i$ and $\displaystyle (\prod _{i\in I}A_i)/(\prod _{i\in I}\alpha _i)$.\end{itemize}\label{prods}\end{remark}

\begin{lemma} If $(A_i)_{i\in I}$ is a non--empty family of algebras from ${\cal C}$, $\displaystyle A=\prod _{i\in I}A_i$ and $\displaystyle {\rm Con}(A)=\{\prod _{i\in I}\theta _i\ |\ (\forall \, i\in I)\, (\theta _i\in {\rm Con}(A_i))$, then:\begin{enumerate}
\item\label{specprod1} $\displaystyle {\rm Con}_2(A)=\bigcup _{i\in I}\{\theta _i\times \prod _{j\in I\setminus \{i\}}\nabla _{A_j}\ |\ \theta _i\in {\rm Con}_2(A_i)\}$;
\item\label{specprod2} $\displaystyle {\rm Max}(A)=\bigcup _{i\in I}\{\theta _i\times \prod _{j\in I\setminus \{i\}}\nabla _{A_j}\ |\ \theta _i\in {\rm Max}(A_i)\}$; ${\rm Max}(A)={\rm Con}_2(A)$ iff ${\rm Max}(A_i)={\rm Con}_2(A_i)$ for all $i\in I$;
\item\label{specprod3} if ${\cal C}$ is congruence--modular and, for any families $(\alpha _i)_{i\in I}$ and $(\beta _i)_{i\in I}$ such that $\alpha _i,\beta _i\in {\rm Con}(A_i)$ for all $i\in I$, we have $\displaystyle [\prod _{i\in I}\alpha _i,\prod _{i\in I}\beta _i]_A=\prod _{i\in I}[\alpha _i,\beta _i]_{A_i}$, then: $\displaystyle {\rm Spec}(A)=\bigcup _{i\in I}\{\theta _i\times \prod _{j\in I\setminus \{i\}}\nabla _{A_j}\ |\ \theta _i\in {\rm Spec}(A_i)\}$; ${\rm Spec}(A)={\rm Con}_2(A)$ iff ${\rm Spec}(A_i)={\rm Con}_2(A_i)$ for all $i\in I$; ${\rm Spec}(A)={\rm Max}(A)$ iff ${\rm Spec}(A_i)={\rm Max}(A_i)$ for all $i\in I$.\end{enumerate}\label{specprod}\end{lemma}

\begin{proof} (\ref{specprod1}) By the form of ${\rm Con}(A)$ stated in the enunciation and the fact that, by Remark \ref{prods}, for any $\displaystyle (\theta _j)_{j\in I}\in \{\prod _{j\in I}{\rm Con}(A_j)$, we have: $\displaystyle \prod _{j\in I}\theta _j\in {\rm Con}_2(A)$ iff, for some $i\in I$, $\theta _i\in {\rm Con}_2(A_i)$ and, for all $j\in I\setminus \{i\}$, $\theta _j\in {\rm Con}_1(A_j)=\{\nabla _{A_j}\}$.

\noindent (\ref{specprod2}) Let $i\in I$, $\theta _i\in {\rm Max}(A_i)\subseteq {\rm Con}(A_i)\setminus \{\nabla _{A_i}\}$, and $\displaystyle \theta =\theta _i\times \prod _{j\in I\setminus \{i\}}\nabla _{A_j}\in {\rm Con}(A)\setminus \{\nabla _A\}$. Let $\mu \in {\rm Con}(A)$ such that $\theta \subsetneq \mu $. Then, by the hypothesis on the form of ${\rm Con}(A)$, $\displaystyle \mu =\prod _{j\in I}\mu _j$ for some $\displaystyle (\mu _j)_{j\in I}\in \{\prod _{j\in I}{\rm Con}(A_j)$. Since $\theta \subsetneq \mu $, it follows that $\mu _j=\nabla _{A_j}$ for all $j\in I\setminus \{i\}$, and $\mu _i\supsetneq \theta _i\in {\rm Max}(A_i)$, hence $\mu _i=\nabla _{A_i}$. Thus $\displaystyle \mu =\prod _{j\in I}\nabla _{A_j}=\nabla _A$. Therefore $\theta \in {\rm Max}(A)$.

Now let $\theta \in {\rm Max}(A)$, so that $\displaystyle \theta =\prod _{j\in I}\theta _j$ for some $\displaystyle (\theta _j)_{j\in I}\in \{\prod _{j\in I}{\rm Con}(A_j)$. Then $\theta \neq \nabla _A$, so there exists an $i\in I$ such that $\theta _i\neq \nabla _{A_i}$. Assume by absurdum that $\theta _i\notin {\rm Max}(A_i)$, so that there exists a $\mu _i\in {\rm Con}(A_i)$ with $\theta _i\subsetneq \mu _i\subsetneq \nabla _{A_i}$. Then $\displaystyle \theta =\prod _{j\in I}\theta _j\subsetneq \mu _i\times \prod _{j\in I\setminus \{i\}}\theta _j\subseteq \mu _i\times \prod _{j\in I\setminus \{i\}}\nabla _{A_j}\subsetneq \prod _{j\in I}\nabla _{A_j}=\nabla _A$, which contradicts the fact that $\theta \in {\rm Max}(A)$. So $\theta _i\in {\rm Max}(A_i)$. Now assume by absurdum that there exists a $k\in I\setminus \{i\}$ such that $\theta _k\neq \nabla _{A_k}$. Then $\displaystyle \theta =\prod _{j\in I}\theta _j\subsetneq \nabla _{A_i}\times \prod _{j\in I\setminus \{i\}}\theta _j\subsetneq \prod _{j\in I}\nabla _{A_j}=\nabla _A$, which contradicts the fact that $\theta \in {\rm Max}(A)$. Hence $\displaystyle \theta =\theta _i\times \prod _{j\in I\setminus \{i\}}\nabla _{A_j}$.

By (\ref{specprod1}), we get the second statement.

\noindent (\ref{specprod3}) Let $i\in I$, $\theta _i\in {\rm Spec}(A_i)\subseteq {\rm Con}(A_i)\setminus \{\nabla _{A_i}\}$, and $\displaystyle \theta =\theta _i\times \prod _{j\in I\setminus \{i\}}\nabla _{A_j}\in {\rm Con}(A)\setminus \{\nabla _A\}$. Let $\alpha ,\beta \in {\rm Con}(A)$ such that $[\alpha ,\beta ]_A\subseteq \theta $. Then $\displaystyle \alpha =\prod _{j\in I}\alpha _j$ and $\displaystyle \beta =\prod _{j\in I}\beta _j$ for some $\displaystyle (\alpha _j)_{j\in I},(\beta _j)_{j\in I}\in \prod _{j\in I}{\rm Con}(A_j)$. Then $\displaystyle \theta =\theta _i\times \prod _{j\in I\setminus \{i\}}\nabla _{A_j}\supseteq [\alpha ,\beta ]_A=[\prod _{j\in I}\alpha _j,\prod _{j\in I}\beta _j]_{A_j}=\prod _{j\in I}[\alpha _j,\beta _j]_{A_j}$, thus $[\alpha _i,\beta _i]_{A_i}\subseteq \theta _i\in {\rm Spec}(A_i)$, hence $\alpha _i\subseteq \theta _i$ or $\beta _i\subseteq \theta _i$. Hence $\displaystyle \alpha =\prod _{j\in I}\alpha _j\subseteq =\theta _i\times \prod _{j\in I\setminus \{i\}}\nabla _{A_j}=\theta $ and $\displaystyle \beta =\prod _{j\in I}\beta _j\subseteq =\theta _i\times \prod _{j\in I\setminus \{i\}}\nabla _{A_j}=\theta $, therefore $\theta \in {\rm Spec}(A)$.

Now let $\theta \in {\rm Spec}(A)\subseteq {\rm Con}(A)\setminus \{\nabla _A\}$, so that $\displaystyle \theta =\prod _{j\in I}\theta _j$ for some $\displaystyle (\theta _j)_{j\in I}\in \{\prod _{j\in I}{\rm Con}(A_j)$ and there exists an $i\in I$ such that $\theta _i\neq \nabla _{A_i}$. Assume by absurdum that $\theta _i\notin {\rm Spec}(A_i)$, so that there exist $\alpha _i,\beta _i\in {\rm Con}(A_i)$ such that $[\alpha _i,\beta _i]_{A_i}\subseteq \theta _i$, but $\alpha _i\nsubseteq \theta _i$ and $\beta _i\nsubseteq \theta _i$. Let $\displaystyle \alpha =\alpha _i\times \prod _{j\in I\setminus \{i\}}\theta _j\in {\rm Con}(A)$ and $\displaystyle \beta =\beta _i\times \prod _{j\in I\setminus \{i\}}\nabla _{A_j}\in {\rm Con}(A)$. Then $\alpha \nsubseteq \theta $ and $\beta \nsubseteq \theta $, but $\displaystyle [\alpha ,\beta ]_A=[\alpha _i,\beta _i]_{A_i}\times \prod _{j\in I\setminus \{i\}}[\theta _j,\nabla _{A_j}]_{A_j}=[\alpha _i,\beta _i]_{A_i}\times \prod _{j\in I\setminus \{i\}}\theta _j\subseteq \prod _{j\in I}\theta _j=\theta $, which contradicts the fact that $\theta \in {\rm Spec}(A)$. Hence $\theta _i\in {\rm Spec}(A_i)$. Now assume by absurdum that there exists a $k\in I\setminus \{i\}$ such that $\theta _k\neq \nabla _{A_k}$. Let $\displaystyle \alpha =\nabla _{A_i}\times \prod _{j\in I\setminus \{i\}}\theta _j\in {\rm Con}(A)$ and $\displaystyle \beta =\theta _i\times \prod _{j\in I\setminus \{i\}}\nabla _{A_j}\in {\rm Con}(A)$. Since $\nabla _{A_i}\nsubseteq \theta _i$, we have $\alpha \nsubseteq \theta $; since $\nabla _{A_k}\nsubseteq \theta _k$, we have $\beta \nsubseteq \theta $. But $\displaystyle [\alpha ,\beta ]_A=[\nabla _{A_i},\theta _i]_{A_i}\times \prod _{j\in I\setminus \{i\}}[\theta _j,\nabla _{A_j}]_{A_j}=\prod _{j\in I}\theta _j=\theta \subseteq \theta $, which contradicts the fact that $\theta \in {\rm Spec}(A)$. Hence $\displaystyle \theta =\theta _i\times \prod _{j\in I\setminus \{i\}}\nabla _{A_j}$.

By (\ref{specprod1}), we get the second statement. By (\ref{specprod2}), we get the third statement.\end{proof}

\begin{proposition} If ${\cal C}$ is congruence--modular and fulfills the equivalent conditions from Proposition \ref{prodcongr}, then, for any $n\in \N ^*$ and any algebras $A_1,\ldots ,A_n$ from ${\cal C}$, if $\displaystyle A=\prod _{i=1}^nA_i$, then: \begin{enumerate}
\item\label{specprodfin1} $\displaystyle {\rm Con}_2(A)=\bigcup _{i=1}^n\{\theta _i\times \prod _{j\in \overline{1,n}\setminus \{i\}}\nabla _{A_j}\ |\ \theta _i\in {\rm Con}_2(A_i)\}$;
\item\label{specprodfin2} $\displaystyle {\rm Max}(A)=\bigcup _{i=1}^n\{\theta _i\times \prod _{j\in \overline{1,n}\setminus \{i\}}\nabla _{A_j}\ |\ \theta _i\in {\rm Max}(A_i)\}$; ${\rm Max}(A)={\rm Con}_2(A)$ iff ${\rm Max}(A_i)={\rm Con}_2(A_i)$ for all $i\in \overline{1,n}$;
\item\label{specprodfin3} $\displaystyle {\rm Spec}(A)=\bigcup _{i=1}^n\{\theta _i\times \prod _{j\in \overline{1,n}\setminus \{i\}}\nabla _{A_j}\ |\ \theta _i\in {\rm Spec}(A_i)\}$; ${\rm Spec}(A)={\rm Con}_2(A)$ iff ${\rm Spec}(A_i)={\rm Con}_2(A_i)$ for all $i\in \overline{1,n}$; ${\rm Spec}(A)={\rm Max}(A)$ iff ${\rm Spec}(A_i)={\rm Max}(A_i)$ for all $i\in \overline{1,n}$.\end{enumerate}\label{specprodfin}\end{proposition}

\begin{proof} By Lemma \ref{specprod} and Proposition \ref{comutprod}.\end{proof}

\begin{corollary} If ${\cal C}$ is congruence--modular and fulfills the equivalent conditions from Proposition \ref{prodcongr}, then, for any $n\in \N ^*$, if $A_i$ and $B_i$ are algebras from ${\cal C}$ and $f_i:A_i\rightarrow B_i$ is a morphism in ${\cal C}$ for every $i\in \overline{1,n}$, then:\begin{enumerate}
\item\label{corprod1} $\displaystyle \prod _{i=1}^nf_i$ is Max--admissible iff $f_1,\ldots ,f_n$ are Max--admissible;
\item\label{corprod2} $\displaystyle \prod _{i=1}^nf_i$ is admissible iff $f_1,\ldots ,f_n$ are admissible.\end{enumerate}\label{corprod}\end{corollary}

\begin{remark} Statement (\ref{corprod1}) in Corollary \ref{corprod} also holds for an arbitrary non--empty index set $I$ instead of $\overline{1,n}$, if both $\displaystyle \prod _{i\in I}A_i$ and $\displaystyle \prod _{i\in I}B_i$ have the congruences as in the enunciation of Lemma \ref{specprod}; if, furthermore, both $\displaystyle \prod _{i\in I}A_i$ and $\displaystyle \prod _{i\in I}B_i$  fulfill the condition on the commutator from Lemma \ref{specprod}, (\ref{specprod3}), then statement (\ref{corprod2}) in Corollary \ref{corprod} holds, as well.\end{remark}

\begin{corollary} If ${\cal C}$ is congruence--modular and semi--degenerate or congruence--distributive, then ${\cal C}$ fulfills the properties from Proposition \ref{specprodfin} and Corollary \ref{corprod}.\label{distribsemid}\end{corollary}

\begin{proof} By Remark \ref{distribprod} and Lemma \ref{semidprod}.\end{proof}

For any lattices $L$ and $M$ such that $L$ has a $1$ and $M$ has a $0$, we shall denote by $L\oplus M$ the ordinal sum of $L$ with $M$ and, for any $\alpha \in {\rm Con}(L)$ and any $\beta \in {\rm Con}(M)$, by $\alpha \oplus \beta =eq((L/\alpha \setminus \{c/\alpha \})\cup (M/\beta \setminus \{c/\beta \})\cup \{c/\alpha\cup c/\beta \})$, where $c$ is the common element of $L$ and $M$ in $L\oplus M$.

\begin{lemma}{\rm \cite{fclp},\cite{eudacs15}} For any lattices $L$ and $M$ such that $L$ has a $1$ and $M$ has a $0$, ${\rm Con}(L\oplus M)=\{\alpha \oplus \beta \ |\ \alpha \in {\rm Con}(L),\beta \in {\rm Con}(M)\}\cong {\rm Con}(L)\times {\rm Con}(M)$.\label{ordsumdacs}\end{lemma}

\begin{lemma} For any lattices $L$ and $M$ such that $L$ has a $1$ and $M$ has a $0$, any $\alpha \in {\rm Con}(L)$ and any $\beta \in {\rm Con}(M)$, $|(L\oplus M)/(\alpha \oplus \beta )|=|L/\alpha |+|M/\beta |-1$.\label{cardsumcongr}\end{lemma}

\begin{proof} Clear, from the definition of $\alpha \oplus \beta $.\end{proof}

\begin{proposition} Let $n\in \N ^*$, $L_1,\ldots ,L_n$ be bounded lattices and $\displaystyle L=\bigoplus _{i=1}^nL_i$. Then:\begin{enumerate}
\item\label{ordsum1} $\displaystyle {\rm Con}(L)=\{\bigoplus _{i=1}^n\theta _i\ |\ (\forall \, i\in \overline{1,n})\, (\theta _i\in {\rm Con}(L_i))\}\cong \prod _{i=1}^n{\rm Con}(L_i)$;
\item\label{ordsum0} for all $\theta _1\in {\rm Con}(L_1),\ldots ,\theta _n\in {\rm Con}(L_n)$, if $\displaystyle \theta =\bigoplus _{i=1}^n\theta _i$, then $\displaystyle |L/\theta |=(\sum _{i=1}^n|L_i/\theta _i|)-n+1$;
\item\label{ordsum2} $\displaystyle {\rm Con}_2(L)=\bigcup _{i=1}^n\{\nabla _{L_1}\oplus \nabla _{L_2}\oplus \ldots \oplus \nabla _{L_{i-1}}\oplus \theta _i\oplus \nabla _{L_{i+1}}\oplus \ldots \oplus \nabla _{L_n}\ |\ \theta _i\in {\rm Con}_2(L_i)\}$;
\item\label{ordsum3} $\displaystyle {\rm Max}(L)=\bigcup _{i=1}^n\{\nabla _{L_1}\oplus \nabla _{L_2}\oplus \ldots \oplus \nabla _{L_{i-1}}\oplus \theta _i\oplus \nabla _{L_{i+1}}\oplus \ldots \oplus \nabla _{L_n}\ |\ \theta _i\in {\rm Max}(L_i)\}$; ${\rm Max}(L)={\rm Con}_2(L)$ iff ${\rm Max}(L_i)={\rm Con}_2(L_i)$ for all $i\in \overline{1,n}$;
\item\label{ordsum4} $\displaystyle {\rm Spec}(L)=\bigcup _{i=1}^n\{\nabla _{L_1}\oplus \nabla _{L_2}\oplus \ldots \oplus \nabla _{L_{i-1}}\oplus \theta _i\oplus \nabla _{L_{i+1}}\oplus \ldots \oplus \nabla _{L_n}\ |\ \theta _i\in {\rm Spec}(L_i)\}$; ${\rm Spec}(L)={\rm Con}_2(L)$ iff ${\rm Spec}(L_i)={\rm Con}_2(L_i)$ for all $i\in \overline{1,n}$;  ${\rm Spec}(L)={\rm Max}(L)$ iff ${\rm Spec}(L_i)={\rm Max}(L_i)$ for all $i\in \overline{1,n}$.\end{enumerate}\label{ordsum}\end{proposition}

\begin{proof} (\ref{ordsum1}) By Lemma \ref{ordsumdacs}.

\noindent (\ref{ordsum0}) By Lemma \ref{cardsumcongr}.

\noindent (\ref{ordsum2}) By (\ref{ordsum1}) and (\ref{ordsum0}).

\noindent (\ref{ordsum3}) and (\ref{ordsum4}) follow from (\ref{ordsum1}) and (\ref{ordsum2}) through a straightforward proof similar to the one for Lemma \ref{specprod}.\end{proof}

If $L$, $L^{\prime }$, $M$ and $M^{\prime }$ are bounded lattices and $h:L\rightarrow M$ and $h^{\prime }:L^{\prime }\rightarrow M^{\prime }$ are bounded lattice morphisms, then we shall denote by $h\oplus h^{\prime }:L\oplus L^{\prime }\rightarrow M\oplus M^{\prime }$ the function defined by: for all $x\in L\oplus L^{\prime }$, $(h\oplus h^{\prime })(x)=\begin{cases}h(x), & x\in L,\\ h^{\prime }(x), & x\in L^{\prime }.\end{cases}$ Clearly, $h\oplus h^{\prime }$ is a bounded lattice morphism.

\begin{corollary} Let $n\in \N ^*$, $L_i$ and $M_i$ be bounded lattices and $h_i:L_i\rightarrow M_i$ be a bounded lattice morphism for every $i\in \overline{1,n}$, and $\displaystyle h=\bigoplus _{i=1}^nh_i:\bigoplus _{i=1}^nL_i\rightarrow \bigoplus _{i=1}^nM_i$. Then:\begin{itemize}
\item $h$ is Max--admissible iff $h_1,\ldots ,h_n$ are Max--admissible;
\item $h$ is admissible iff $h_1,\ldots ,h_n$ are admissible.\end{itemize}\label{corsumord}\end{corollary}

\begin{remark} The statements in Proposition \ref{ordsum} and Corollary \ref{corsumord} hold even if $L_1$ and $M_1$ do not have a $0$ and $L_n$ and $M_n$ do not have a $1$.

The statements in Proposition \ref{ordsum} also hold for any bounded orderred structures whose ordinal sums have congruences exactly of the form in Lemma \ref{ordsumdacs}; if such structures have the property that the ordinal sum between two morphisms, defined as above, is again a morphism, then they also fulfill the statements in Corollary \ref{corsumord}.\end{remark}

\section{Prime, Maximal and Two--class Congruences in Particular Kinds of Lattices and Universal Algebras}
\label{speclat}

In this section we point out certain kinds of lattices and congruence--distributive algebras in which either the prime congruences coincide to the maximal ones, or the maximal congruences coincide to the two--class ones, or both of these relationships hold. Such algebras are important for the study of admissible and Max--admissible morphisms. 

Let $M$ be a non--empty set. Clearly, ${\rm Eq}_0(M)=\emptyset $, ${\rm Eq}_1(M)=\{\Delta _M\}$ and ${\rm Eq}_{\kappa }(M)\cap {\rm Eq}_{\lambda }(M)=\emptyset $ for any cardinality $\lambda \neq \kappa $. It is well known and immediate that, for all $\phi ,\psi \in {\rm Eq}(M)$, $\phi \subseteq \psi $ iff $\phi $ is a refinement of $\psi $, that is each class of $\psi $ is a union of classes of $\phi $, which implies $|M/\psi |\leq |M/\phi |$; clearly, if $\phi \subseteq \psi $ and $M/\phi $ is finite, then: $\phi =\psi $ iff $|M/\phi |=|M/\psi |$.

\begin{remark} Clearly, ${\rm Con}_2(A)\subseteq {\rm Max}(A)$. Indeed, if $\mu \in {\rm Con}_2(A)$, then $\mu \notin {\rm Con}_1(A)=\{\nabla _A\}$, so $\mu $ is a proper congruence of $A$, and, for any $\theta \in {\rm Con}(A)$, $\mu \subsetneq \theta $ iff $\theta \in {\rm Con}_1(A)=\{\nabla _A\}$ iff $\theta =\nabla _A$; therefore $\mu \in {\rm Max}(A)$.\label{2max}\end{remark}

In the case of congruence--distributive varieties, the second statement in Lemma \ref{folclor} holds even without $\nabla _A$ being finitely generated:

\begin{lemma} If ${\cal C}$ is congruence--distributive, or ${\cal C}$ is congruence--modular, $A$ is congruence--distributive and the commutator in $A$ equals the intersection of congruences, then ${\rm Max}(A)\subseteq {\rm Spec}(A)$.\label{specdistrib}\end{lemma}

\begin{proof} Let $\theta \in {\rm Max}(A)$. Assume by absurdum that there exist $\alpha ,\beta \in {\rm Con}(A)$ such that $[\alpha ,\beta ]_A=\alpha \cap \beta \subseteq \theta $, but $\alpha \nsubseteq \theta $ and $\beta \nsubseteq \theta $. Then $\alpha \vee \theta \neq \theta $ and $\beta \vee \theta \neq \theta $, so, since $\theta \subseteq \alpha \vee \theta $ and $\theta \subseteq \beta \vee \theta $, we have $\theta \subsetneq \alpha \vee \theta $ and $\theta \subsetneq \beta \vee \theta $. $\theta \in {\rm Max}(A)$, hence $\alpha \vee \theta =\beta \vee \theta =\nabla _A$, thus $\theta =(\alpha \cap \beta )\vee \theta =(\alpha \vee \theta )\cap (\beta \vee \theta )=\nabla _A\cap \nabla _A=\nabla _A$, which is a contradiction to $\theta \in {\rm Max}(A)$. Therefore $\theta \in {\rm Spec}(A)$.\end{proof}

\begin{remark} Notice, from Example \ref{fiectip}, that ${\rm Max}({\cal P})={\rm Con}_2({\cal P})$.\label{2pentagon}\end{remark}

\begin{proposition} If $B$ is a Boolean algebra, then ${\rm Spec}(B)={\rm Max}(B)={\rm Con}_2(B)$.\label{specbool}\end{proposition}

\begin{proof} There are many ways to prove this statement. One way is to use the well--known fact that ${\rm Spec}_{\rm Filt}(B)={\rm Max}_{\rm Filt}(B)=\{F\in {\rm Filt}(B)\ |\ B/F=B/\varphi _B(F)\cong {\cal L}_2\}$ and the fact that $\varphi _B:{\rm Filt}(B)\rightarrow {\rm Con}(B)$ is a bounded lattice isomorphism, thus, by Remark \ref{primdistrib}, ${\rm Spec}(B)=\varphi _B({\rm Spec}_{\rm Filt}(B))=\varphi _B({\rm Max}_{\rm Filt}(B))={\rm Max}(B)=\{\varphi _B(F)\ |\ F\in {\rm Filt}(B),B/\varphi _B(F)\cong {\cal L}_2\}=\{\phi \in {\rm Con}(B)\ |\ B/\phi \cong {\cal L}_2\}={\rm Con}_2(B)$.\end{proof}

\begin{theorem}{\rm \cite[Theorem 8.15, p. 128]{blyth}} If $L$ is a bounded  distributive lattice, then there exists a Boolean algebra $B$ such that the lattices ${\rm Con}(L)$ and ${\rm Con}(B)$ are isomorphic.\label{distribvsbool}\end{theorem}

\begin{corollary} If $L$ is a bounded distributive lattice, then ${\rm Spec}(L)={\rm Max}(L)={\rm Con}_2(L)$.\label{mysuper}\end{corollary}

\begin{proof} By Proposition \ref{specbool}, Theorem \ref{distribvsbool} and Remark \ref{primdistrib}, ${\rm Spec}(L)={\rm Max}(L)$. By Remark \ref{2max}, ${\rm Con}_2(L)\subseteq {\rm Max}(L)$. Now let $\phi \in {\rm Spec}(L)={\rm Max}(L)$, and let $F=1/\phi \in {\rm Filt}(L)$. Since $\phi \neq \nabla _L$, we have $F\neq L$. Assume by absurdum that $F\notin {\rm Spec}_{\rm Filt}(L)$, so that there are at least two distinct prime filters of $L$ which include $F$. Let $P\in {\rm Spec}_{\rm Filt}(L)$ such that $F\subseteq P$ and $Q$ be the intersection of the prime filters of $L$ which differ from $P$ and include $F$, so that $Q\in {\rm Filt}(L)\setminus \{L\}$ and $F=P\cap G$, thus $\varphi _L(F)=\varphi _L(P)\cap \varphi _L(G)$. Since $F=1/\phi \in L/\phi $, it follows that $\varphi _L(F)\subseteq \phi $, so, by Theorem \ref{distrib}, $[\varphi _L(P),\varphi _L(G)]_L=\varphi _L(P)\cap \varphi _L(G)\subseteq \phi $. But $1/\varphi _L(P)=P\nsubseteq F=1/\phi $ and $1/\varphi _L(G)=G\nsubseteq F=1/\phi $, thus $\varphi _L(P)\nsubseteq /\phi $ and $\varphi _L(G)\nsubseteq /\phi $. This contradicts the primality of $\phi $. Hence $F\in {\rm Spec}_{\rm Filt}(L)$, therefore $\phi \subseteq eq(F,L\setminus F)\in {\rm Con}_2(L)\subseteq {\rm Con}(L)\setminus \{\nabla _L\}$, by Remark \ref{con2lat}. Since $\phi \in {\rm Max}(L)$, it follows that $\phi =eq(F,L\setminus F)\in {\rm Con}_2(L)$. Therefore ${\rm Con}_2(L)\subseteq {\rm Max}(L)={\rm Spec}(L)\subseteq {\rm Con}_2(L)$, hence ${\rm Spec}(L)={\rm Max}(L)={\rm Con}_2(L)$.\end{proof}

Note that, in the case of chains, the boundeness condition in Corollary \ref{mysuper} is not necessary:

\begin{lemma} If $L$ is a chain, then ${\rm Spec}(L)={\rm Max}(L)={\rm Con}_2(L)$.\label{chainprime}\end{lemma}

\begin{proof} Let $L$ be a chain. By Remark \ref{2max} and Lemma \ref{specdistrib}, ${\rm Con}_2(L)\subseteq {\rm Max}(L)\subseteq {\rm Spec}(L)$. Now let $\theta \in {\rm Con}(L)$ such that $|L/\theta |\geq 3$. Then, by Remark \ref{chainconvex}, there exist $C,D,E\in L/\theta $ such that $C<D<E$, so that $M=\{B\in L/\theta \ |\ D<B\}\neq \emptyset $ and $N=\{B\in L/\theta \ |\ B<D\}\neq \emptyset $. Let $\displaystyle \alpha =eq(M\cup \{D\cup \bigcup _{B\in N}B\})$ and $\displaystyle \beta =eq(N\cup \{D\cup \bigcup _{B\in M}B\})$. Then, by Lemma \ref{chainconvex}, $\alpha ,\beta \in {\rm Con}(L)$. Clearly, $\alpha \cap \beta =\theta $, $\alpha \neq \beta $ and, since $|L/\alpha |\geq 2$ and $|L/\beta |\geq 2$, we have $\alpha \neq \nabla _L$ and $\beta \neq \nabla _L$; thus $\theta \subsetneq \alpha \subsetneq \nabla _L$, so $\theta \notin {\rm Max}(L)$. Also, by Theorem \ref{distrib}, $[\alpha ,\beta ]_L=\alpha \cap \beta =\theta \subseteq \theta $, but, clearly, $\alpha \nsubseteq \theta $ and $\beta \nsubseteq \theta $; thus  $\theta \notin {\rm Spec}(L)$. Since ${\rm Spec}(L)\subseteq {\rm Con}(L)\setminus \{\nabla _L\}={\rm Con}(L)\setminus {\rm Con}_1(L)$, it follows that ${\rm Spec}(L)\subseteq {\rm Con}_2(L)$. Hence ${\rm Con}_2(L)\subseteq {\rm Max}(L)\subseteq {\rm Spec}(L)\subseteq {\rm Con}_2(L)$, thus ${\rm Spec}(L)={\rm Max}(L)={\rm Con}_2(L)$.\end{proof}

\begin{theorem}{\rm \cite[Theorem 3.5.1, p. 75]{schmidt}} If $L$ is a finite modular lattice, then ${\rm Con}(L)$ is a Boolean algebra.\label{1latconbool}\end{theorem}

\begin{theorem}{\rm \cite[p. 80]{cwdw}} If $L$ is a relatively complemented lattice fulfilling the ACC, then ${\rm Con}(L)$ is a Boolean algebra.\label{2latconbool}\end{theorem}

\begin{remark} By Remark \ref{distribprod}, Proposition \ref{ordsum}, (\ref{ordsum1}), and Theorems \ref{1latconbool} and \ref{2latconbool}, if a lattice $L$ can be obtained through finite direct products and/or finite ordinal sums from finite modular lattices and relatively complemented lattices fulfilling the ACC, then ${\rm Con}(L)$ is a Boolean algebra.\label{auconbool}\end{remark}

\begin{corollary} If $L$ is a finite modular lattice or a relatively complemented lattice fulfilling the ACC, then ${\rm Spec}(L)={\rm Max}(L)$.\label{spec=max}\end{corollary}

\begin{proof} By Proposition \ref{conbool} and Theorems \ref{1latconbool} and \ref{2latconbool}.\end{proof}

\begin{proposition} Let $L$ be a lattice.

\begin{enumerate}
\item\label{spec2cls1} If $L$ can be obtained through finite direct products and/or finite ordinal sums from chains and/or bounded distributive lattices, then ${\rm Spec}(L)={\rm Max}(L)={\rm Con}_2(L)$.
\item\label{spec2cls2} If $L$ can be obtained through finite direct products and/or finite ordinal sums from chains and/or bounded distributive lattices and/or finite modular lattices and/or relatively complemented lattices fulfilling the ACC, then ${\rm Spec}(L)={\rm Max}(L)$.
\item\label{spec2cls3} If $L$ can be obtained through finite direct products and/or finite ordinal sums from chains and/or bounded distributive lattices and/or the pentagon, then ${\rm Max}(L)={\rm Con}_2(L)$.\end{enumerate}\label{spec2cls}\end{proposition}

\begin{proof} By Remark \ref{distribprod}, Corollary \ref{distribsemid}, Proposition \ref{ordsum}, Remark \ref{2pentagon}, Corollary \ref{mysuper}, Lemma \ref{chainprime} and Corollary \ref{spec=max}.\end{proof}

\begin{example} Clearly, the cathegories of lattices pointed out above are not exhaustive for the properties they illustrate. Here is a finite lattice $N$ which is neither modular, nor relatively complemented, in fact which can not be obtained through either of the constructions in Proposition \ref{spec2cls}, but whose lattice of congruences is Boolean: ${\rm Con}(N)=\{\Delta _N,\nabla _N\}\cong {\cal L}_2$:
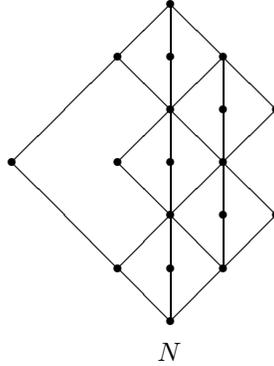
\vspace*{-17pt}

\begin{center}
\begin{picture}(50,150)(0,0)
\put(30,25){\line(0,1){120}}
\put(30,25){\line(1,1){40}}
\put(70,65){\line(-1,1){20}}
\put(50,85){\line(1,1){20}}
\put(70,65){\circle*{3}}
\put(50,65){\circle*{3}}
\put(-30,85){\circle*{3}}
\put(30,25){\line(-1,1){60}}
\put(30,65){\line(1,-1){20}}
\put(30,65){\line(-1,-1){20}}
\put(30,65){\line(1,1){20}}

\put(30,65){\line(-1,1){20}}
\put(30,105){\line(1,-1){20}}
\put(30,105){\line(-1,-1){20}}
\put(30,105){\line(1,1){20}}
\put(30,105){\line(-1,1){20}}
\put(30,145){\line(1,-1){40}}
\put(70,105){\circle*{3}}
\put(50,105){\circle*{3}}
\put(30,145){\line(-1,-1){60}}
\put(50,45){\line(0,1){80}}
\put(30,25){\circle*{3}}
\put(10,45){\circle*{3}}

\put(30,45){\circle*{3}}
\put(50,45){\circle*{3}}
\put(10,85){\circle*{3}}
\put(30,85){\circle*{3}}
\put(50,85){\circle*{3}}
\put(10,125){\circle*{3}}
\put(30,125){\circle*{3}}
\put(50,125){\circle*{3}}
\put(30,65){\circle*{3}}
\put(30,105){\circle*{3}}
\put(30,145){\circle*{3}}
\put(25,10){$N$}
\end{picture}
\end{center}\vspace*{-17pt}

Many examples of such lattices can be constructed, and, of course, any such lattice can be inserted in the construction of the lattice $L$ from Proposition \ref{spec2cls}, (\ref{spec2cls2}), as well as in the construction from Remark \ref{auconbool}. See in \cite{bal}, \cite{birkhoff}, \cite{blyth}, \cite{cwdw}, \cite{gratzer}, \cite{schmidt} more types of lattices whose lattices of congruences are Boolean.\label{latn}\end{example}

\section{On Cardinalities of Quotient Sets and Direct and Inverse Images of Congruences Through Morphisms}
\label{throughmorph}

We start this section with a result concerning the cardinalities of quotient lattices through maximal congruences. Then we compare the cardinalities of quotient sets through equivalence relations with the cardinalities of quotient sets through the direct and inverse images of those equivalence relations through functions; we prove that result for sets, functions and equivalence relations because it does not need supplementary hypotheses; clearly, when applied to algebras and morphisms, it will give analogous results on congruences. Then we apply this result to morphisms and two--class congruences and we obtain more results on admissibility and Max--admissibility, also using the results in the previous sections.

\begin{proposition} Let $\kappa $ be a cardinal number. Then: there exists a lattice $L$ and a $\mu \in {\rm Max}(L)$ with $|L/\mu |=\kappa $ iff $\kappa =2$ or $\kappa \geq 5$.\label{cardquomax}\end{proposition}

\begin{proof} Let $L$ be a lattice. Then ${\rm Con}_0(L)=\emptyset $ and ${\rm Con}_1(L)=\{\nabla _L\}$, which contain no proper, thus no maximal congruence of $L$. Proposition \ref{spec2cls}, (\ref{spec2cls3}), provides us with an infinity of examples of lattices having maximal congruences which determine quotient lattices of cardinality $2$.

Now let $\kappa \geq 5$. Let $M$ be a set with $|M|=\kappa -2\geq 3$ and $0$, $1$ be two elements which fulfill: $0,1\notin M$ and $0\neq 1$. Denote $L=M\cup \{0,1\}$ and $\leq =\{(0,a),(a,1)\ |\ a\in M\}$. Then $\leq $ is an order on $L$ and $(L,\leq )$ is a bounded modular non--distributive lattice, whose Hasse diagram we sketch here:
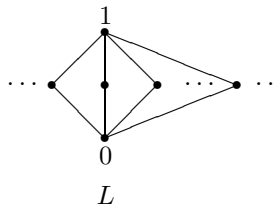
\vspace*{-9pt}

\begin{center}
\begin{picture}(40,80)(0,0)
\put(30,25){\line(0,1){40}}
\put(30,25){\line(1,1){20}}
\put(30,25){\line(-1,1){20}}
\put(30,65){\line(1,-1){20}}
\put(30,65){\line(-1,-1){20}}
\put(30,25){\circle*{3}}
\put(10,45){\circle*{3}}
\put(30,45){\circle*{3}}
\put(50,45){\circle*{3}}
\put(60,45){$\ldots $}
\put(-7,45){$\ldots $}
\put(87,45){$\ldots $}
\put(80,45){\circle*{3}}
\put(30,65){\circle*{3}}
\put(30,65){\line(5,-2){50}}
\put(30,25){\line(5,2){50}}
\put(27,0){$L$}
\put(28,15){$0$}
\put(28,68){$1$}
\end{picture}\end{center}\vspace*{-5pt}

Clearly, ${\rm Con}(L)=\{\Delta _L,\nabla _L\}$, thus ${\rm Max}(L)={\rm Spec}(L)=\{\Delta _L\}$, and $|L/\Delta _L|=|L|=|M|+2=\kappa $.

Now let $L$ be again an arbitrary lattice, and let $\phi \in {\rm Con}_3(L)\cup {\rm Con}_4(L)$, so that $L/\phi $ is isomorphic to ${\cal L}_3$, ${\cal L}_4$ or ${\cal L}_2^2$. If $L/\phi \cong {\cal L}_2^2$, then, by Example \ref{fiectip}, ${\rm Con}(L/\phi )\cong {\rm Con}({\cal L}_2^2)\cong {\cal L}_2^2\ncong {\cal L}_2$, hence $\phi \notin {\rm Max}(L)$ by Remark \ref{congrcat}. Note that ${\cal L}_3\cong {\cal L}_2\oplus {\cal L}_2$ and ${\cal L}_4\cong {\cal L}_2\oplus {\cal L}_2\oplus {\cal L}_2$, hence ${\rm Con}({\cal L}_3)\cong {\rm Con}({\cal L}_2\oplus {\cal L}_2)\cong {\rm Con}({\cal L}_2)^2\cong {\cal L}_2^2$ and ${\rm Con}({\cal L}_4)\cong {\rm Con}({\cal L}_2\oplus {\cal L}_2\oplus {\cal L}_2)\cong {\rm Con}({\cal L}_2)^3\cong {\cal L}_2^3$, by Remark \ref{lant2} and Proposition \ref{ordsum}, (\ref{ordsum1}). Therefore, if $L/\phi \cong {\cal L}_3$ or $L/\phi \cong {\cal L}_4$, then ${\rm Con}(L/\phi )\cong {\rm Con}({\cal L}_3)\cong {\cal L}_2^2\ncong {\cal L}_2$ or ${\rm Con}(L/\phi )\cong {\rm Con}({\cal L}_4)\cong {\cal L}_2^3\ncong {\cal L}_2$, hence $\phi \notin {\rm Max}(L)$ by Remark \ref{congrcat}. Hence $L$ can have no maximal congruences $\mu $ with $|L/\mu |\in \{3,4\}$.\end{proof}

\begin{lemma} Let $M$ and $N$ be non--empty sets, $h:M\rightarrow N$, $\alpha \in {\rm Eq}(M)$, $\beta \in {\rm Eq}(N)$ and $\kappa $ be a cardinal number. Then:\begin{enumerate}
\item\label{cardclase1} $|M/h^*(\beta )|\leq |N/\beta |$ and $\displaystyle h^*({\rm Eq}_{\kappa }(N))\subseteq \bigcup _{c\leq \kappa }{\rm Eq}_c(M)$; 
\item\label{cardclase0} if $h$ is surjective, then: $|M/h^*(\beta )|=|N/\beta |$, $|M/\alpha |\geq |N/h(\alpha )|$, $h^*({\rm Eq}_{\kappa }(N))\subseteq {\rm Eq}_{\kappa }(M)$ and $\displaystyle h({\rm Eq}_{\kappa }(M))\subseteq \bigcup _{c\leq \kappa }{\rm Eq}_c(N)$;
\item\label{cardclase2} if $h$ is bijective, then : $|M/h^*(\beta )|=|N/\beta |$, $|M/\alpha |=|N/h(\alpha )|$, $h^*({\rm Eq}_{\kappa }(N))={\rm Eq}_{\kappa }(M)$ and $h({\rm Eq}_{\kappa }(M))={\rm Eq}_{\kappa }(N)$;
\item\label{cardclase3} $h^*({\rm Eq}_2(N))\subseteq \{\nabla _M\}\cup {\rm Eq}_2(M)$;
\item\label{cardclase4} if $(h^*)^{-1}(\{\nabla _M\})=\{\nabla _N\}$, then $h^*({\rm Eq}_2(N))\subseteq {\rm Eq}_2(M)$; 
\item\label{cardclase5} if $h$ is surjective, then $(h^*)^{-1}(\{\nabla _M\})=\{\nabla _N\}$, $h^*({\rm Eq}_2(N))\subseteq {\rm Eq}_2(M)$ and $h({\rm Eq}_2(M))={\rm Eq}_2(N)$.\end{enumerate}\label{cardclase}\end{lemma}

\begin{proof} (\ref{cardclase1}) Let $\varphi :M/h^*(\beta )\rightarrow N/\beta $, for all $a\in M$, $\varphi (a/h^*(\beta ))=h(a)/\beta $. For every $a,b\in M$, we have: $a/h^*(\beta )=b/h^*(\beta )$ iff $(a,b)\in h^*(\beta )$ iff $(h(a),h(b))\in \beta $ iff $h(a)/\beta =h(b)/\beta $ iff $\varphi (a/h^*(\beta ))=\varphi (b/h^*(\beta ))$; therefore $\varphi $ is well defined and injective, thus $|M/h^*(\beta )|\leq |N/\beta |$. Hence the inclusion in the enunciation.

\noindent (\ref{cardclase0}) If $h$ is surjective, then, clearly, the map $\varphi $ from (\ref{cardclase1}) is surjective, as well, hence $\varphi $ is bijective, therefore $|A/h^*(\beta )|=|B/\beta |$. Also, if $h$ is surjective, then $h(\alpha )\in {\rm Eq}(B)$. Let $\psi :A/\alpha \rightarrow B/h(\alpha )$, for all $a\in A$, $\psi (a/\alpha )=h(a)/h(\alpha )$. For all $x,y\in A$, we have the following: $x/\alpha =y/\alpha $ iff $(x,y)\in \alpha $, which implies $(h(x),h(y))\in h(\alpha )$, which means that $h(x)/h(\alpha )=h(y)/h(\alpha )$, that is $\psi (x/\alpha )=\psi (y/\alpha )$. So $\psi $ is well defined. Since $h$ is surjective, it clearly follows that $\psi $ is surjective, therefore $|A/\alpha |\geq |B/h(\alpha )|$. Hence the inclusions in the enunciation.

\noindent (\ref{cardclase2}) If $h$ is bijective, then $h$ is injective, so $h^*(h(\alpha ))=\alpha $, and $h$ is surjective, so $h(\alpha )\in {\rm Eq}(B)$ and $|A/\alpha |\geq |B/h(\alpha )|$ by (\ref{cardclase0}). Also, $|A/\alpha |=|A/h^*(h(\alpha ))|=|B/h(\alpha )|$, again by (\ref{cardclase0}), in which we take $\beta =h(\alpha )$. Thus $h^*({\rm Eq}_{\kappa }(N))\subseteq {\rm Eq}_{\kappa }(M)$ and $h({\rm Eq}_{\kappa }(M))\subseteq {\rm Eq}_{\kappa }(N)$, therefore, since $h^*(h(\phi ))=\phi $ and $h(h^*(\psi ))=\psi $ for all $\phi \in {\rm Eq}(M)$ and all $\psi \in {\rm Eq}(N)$, it follows that: ${\rm Eq}_{\kappa }(N)=h(h^*({\rm Eq}_{\kappa }(N)))\subseteq h({\rm Eq}_{\kappa }(M))\subseteq {\rm Eq}_{\kappa }(N)$, hence $h({\rm Eq}_{\kappa }(M))={\rm Eq}_{\kappa }(N)$, thus $h^*({\rm Eq}_{\kappa }(N))=h^*(h({\rm Eq}_{\kappa }(M)))={\rm Eq}_{\kappa }(M)$.

\noindent (\ref{cardclase3}) By (\ref{cardclase1}).

\noindent (\ref{cardclase4}) By (\ref{cardclase3}).

\noindent (\ref{cardclase5}) From (\ref{cardclase4}), (\ref{cardclase0}), the fact that $h^*(\nabla _N)=\nabla _M$, $\nabla _M\notin {\rm Eq}_2(M)$ and, if $h$ is surjective, then $h(\nabla _M)=\nabla _N$ and $h^*$ is injective, so $(h^*)^{-1}(\{\nabla _M\})=\{\nabla _N\}$ and $\nabla _N\notin h({\rm Eq}_2(M))$, we obtain: $h^*({\rm Eq}_2(N))\subseteq {\rm Eq}_2(M)$ and $h({\rm Eq}_2(M))\subseteq {\rm Eq}_2(N)$. Since $h$ is surjective, we have $h(h^*(\theta ))=\theta $ for all $\theta \in {\rm Eq}(M)$. Therefore ${\rm Eq}_2(N)=h(h^*({\rm Eq}_2(N)))\subseteq h({\rm Eq}_2(M))\subseteq {\rm Eq}_2(N)$, hence $h({\rm Eq}_2(M))={\rm Eq}_2(N)$.\end{proof}

\begin{remark} Let $L$ and $M$ be bounded lattices, $h:L\rightarrow M$ be a bounded lattice morphism and $\theta \in {\rm Con}(L)$. Then:\begin{itemize}
\item $\theta =\nabla _L$ iff $(0,1)\in \theta $, as shown by Remark \ref{convex};
\item $(h^*)^{-1}(\{\nabla _L\})=\{\nabla _M\}$, because, by the above, for any $\phi \in {\rm Con}(M)$, the following hold: $h^*(\phi )=\Delta _L$ iff $(0,1)\in h^*(\phi )$ iff $(h(0),h(1))\in \phi $ iff $(0,1)\in \phi $ iff $\phi =\nabla _M$.\end{itemize}\label{01}\end{remark}

\begin{lemma} If $L$ and $M$ are bounded lattices and $h:L\rightarrow M$ is a bounded lattice morphism, then:\begin{enumerate}
\item\label{morflat1} $h^*({\rm Con}_2(M))\subseteq {\rm Con}_2(L)$;
\item\label{morflat2} if ${\rm Max}(M)={\rm Con}_2(M)$, then $h$ is Max--admissible;
\item\label{morflat3} if ${\rm Spec}(M)={\rm Max}(M)={\rm Con}_2(M)$, then $h$ is admissible and Max--admissible;
\item\label{morflat4} if ${\rm Spec}(M)={\rm Max}(M)$ and $h$ is Max--admissible, then $h$ is admissible;
\item\label{morflat0} if ${\rm Spec}(L)={\rm Max}(L)$ and $h$ is admissible, then $h$ is Max--admissible;
\item\label{morflat5} if ${\rm Spec}(L)={\rm Max}(L)$ and ${\rm Spec}(M)={\rm Max}(M)$, then: $h$ is admissible iff $h$ is Max--admissible;
\item\label{morflat6} if $M$ can be obtained through finite direct products and/or finite ordinal sums from bounded distributive lattices and/or the pentagon, then $h$ is Max--admissible;
\item\label{morflat7} if $M$ is a bounded distributive lattice, then $h$ is admissible and Max--admissible;
\item\label{morflat8} if $h$ is Max--admissible and $M$ can be obtained through finite direct products and/or finite ordinal sums from bounded distributive lattices and/or finite modular lattices and/or relatively complemented lattices fulfilling the ACC, then $h$ is admissible;
\item\label{morflat9} if $h$ is admissible and $L$ can be obtained through finite direct products and/or finite ordinal sums from bounded distributive lattices and/or finite modular lattices and/or relatively complemented lattices fulfilling the ACC, then $h$ is Max--admissible;
\item\label{morflat10} if both $L$ and $M$ can be obtained through finite direct products and/or finite ordinal sums from bounded distributive lattices and/or finite modular lattices and/or relatively complemented lattices fulfilling the ACC then: $h$ is admissible iff $h$ is Max--admissible.\end{enumerate}\label{morflat}\end{lemma}

\begin{proof} (\ref{morflat1}) By Lemma \ref{cardclase}, (\ref{cardclase4}), and Remark \ref{01}.

\noindent (\ref{morflat2})--(\ref{morflat4}) By the fact that ${\rm Con}_2(L)\subseteq {\rm Max}(L)\subseteq {\rm Spec}(L)$, according to Remark \ref{2max} and Lemma \ref{specdistrib}.

\noindent (\ref{morflat0}) By the fact that ${\rm Con}_2(M)\subseteq {\rm Max}(M)\subseteq {\rm Spec}(M)$, according to Remark \ref{2max} and Lemma \ref{specdistrib}.

\noindent (\ref{morflat5}) Clear.

\noindent (\ref{morflat6}) By (\ref{morflat2}) and Proposition \ref{spec2cls}, (\ref{spec2cls3}).

\noindent (\ref{morflat7}) By (\ref{morflat3}) and Proposition \ref{spec2cls}, (\ref{spec2cls1}).

\noindent (\ref{morflat8}) By (\ref{morflat4}) and Proposition \ref{spec2cls}, (\ref{spec2cls2}).

\noindent (\ref{morflat9}) By (\ref{morflat0}) and Proposition \ref{spec2cls}, (\ref{spec2cls2}).

\noindent (\ref{morflat10}) By (\ref{morflat5}) and Proposition \ref{spec2cls}, (\ref{spec2cls2}).\end{proof}

\begin{proposition} In the class of bounded distributive lattices, all morphisms are admissible and Max--admissible.\label{admdistrib}\end{proposition}

\begin{proof} By Lemma \ref{morflat}, (\ref{morflat7}).\end{proof}

\begin{lemma} If $L$ and $M$ are bounded lattices and $h:L\rightarrow M$ be a bounded lattice morphism. If $h(L)=\{0,1\}$, then $h^*({\rm Con}(M)\setminus \{\nabla _M\})\subseteq {\rm Con}_2(L)$.\label{con2im01}\end{lemma}

\begin{proof} Let $\phi \in {\rm Con}(M)\setminus \{\nabla _M\}$, so that, by Remark \ref{01}, $h^*(\phi )\neq \nabla _L$, thus $(0,1)\notin h^*(\phi )$, which means that $0/h^*(\phi )\neq 1/h^*(\phi )$. Let $a\in L$. Then $h(a)=0$ or $h(a)=1$. If $h(a)=0=h(0)$, then $(h(a),h(0))\in \Delta _M\subseteq \phi $, thus $(a,0)\in h^*(\phi )$, that is $a/h^*(\phi )=0/h^*(\phi )$. Analogously, if $h(a)=1$, then $a/h^*(\phi )=1/h^*(\phi )$. So $L/h^*(\phi )=\{0/h^*(\phi ),1/h^*(\phi )\}$, with $0/h^*(\phi )\neq 1/h^*(\phi )$, therefore $h^*(\phi )\in {\rm Con}_2(L)$.\end{proof}

\begin{proposition} Any bounded lattice morphism whose image is $\{0,1\}$ is admissible and Max--admissible.\label{im01}\end{proposition}

\begin{proof} Let $L$ and $M$ be bounded lattices and $h:L\rightarrow M$ be a bounded lattice morphism with $h(L)=\{0,1\}$. By Lemma \ref{con2im01} and Remarks \ref{2max} and \ref{rfolclor}, if $\phi \in {\rm Spec}(M)\subseteq {\rm Con}(M)\setminus \{\nabla _M\}$ or $\phi \in {\rm Max}(M)\subseteq {\rm Con}(M)\setminus \{\nabla _M\}$, then $h^*(\phi )\in {\rm Con}_2(L)\subseteq {\rm Max}(L)\subseteq {\rm Spec}(L)$, thus $h$ is admissible and Max--admissible.\end{proof}

\begin{remark} The statements in Lemma \ref{con2im01} and Proposition \ref{im01} hold for any morphism between bounded orderred structures of the same type.\end{remark}

\begin{lemma} If $L$ and $M$ are bounded lattices, $M$ is non--trivial and $h:L\rightarrow M$ is a bounded lattice morphism, then $L$ is non--trivial, $h^*({\rm Con}_2(M))\subseteq {\rm Con}_2(L)$ and, if $h$ is surjective, then $h({\rm Con}_2(L))={\rm Con}_2(M)$.\label{lat2cls}\end{lemma}

\begin{proof} Since $h(0)=0\neq 1=h(1)$, it follows that $0\neq 1$ in $L$, thus $L$ is non--trivial. By Lemma \ref{morflat}, (\ref{morflat1}), we have $h^*({\rm Con}_2(M))\subseteq {\rm Con}_2(L)$. By Lemma \ref{cardclase}, (\ref{cardclase4}), if $h$ is surjective, then $h({\rm Con}_2(L))={\rm Con}_2(M)$.\end{proof}

Remark \ref{01} and Lemma \ref{lat2cls} can be generalized:

\begin{lemma} If ${\cal C}$ is semi--degenerate, then:\begin{itemize}
\item $B$ is non--trivial, then $A$ is non--trivial;
\item $(f^*)^{-1}(\{\nabla _A\})=\{\nabla _B\}$ and $f^*({\rm Con}_2(B))\subseteq {\rm Con}_2(A)$; 
\item if $f$ is surjective, then $f({\rm Con}_2(A))={\rm Con}_2(B)$.\end{itemize}\label{gen01}\end{lemma}

\begin{proof} Let $\beta \in {\rm Con}(B)$ and let us define $\varphi :A/f^*(\beta )\rightarrow B/\beta $, for all $a\in M$, $\varphi (a/f^*(\beta ))=h(a)/\beta $. From the proof of Lemma \ref{cardclase}, (\ref{cardclase1}), we get that $\varphi $ is an embedding in ${\cal C}$. Since ${\cal C}$ is semi--degenerate and $A/f^*(\beta )$ is embedded in $B/\beta $, the following equivalences hold: $f^*(\beta )=\nabla _A$ iff $A/f^*(\beta )$ is the trivial algebra iff $B/\beta $ is the trivial algebra iff $\beta =\nabla _B$, therefore $(f^*)^{-1}(\{\nabla _A\})=\{\nabla _B\}$, hence $f^*({\rm Con}_2(B))\subseteq {\rm Con}_2(A)$ by Lemma \ref{cardclase}, (\ref{cardclase4}), and, if $f$ is surjective, then $f({\rm Con}_2(A))={\rm Con}_2(B)$ by Lemma \ref{cardclase}, (\ref{cardclase5}). Since ${\cal C}$ is semi--degenerate and $f(A)$ is embedded in $B$, if $B$ is non--trivial, then $f(A)$ is non--trivial, hence $A$ is non--trivial.\end{proof}

Now let us generalize the statements in Lemma \ref{morflat}.

\begin{corollary}\begin{enumerate} 
\item\label{cardadm1} If ${\rm Max}(B)={\rm Con}_2(B)$ and $(f^*)^{-1}(\{\nabla _A\})=\{\nabla _B\}$, then $f$ is Max--admissible.
\item\label{cardadm2} If ${\cal C}$ is congruence--modular, ${\rm Spec}(B)={\rm Max}(B)={\rm Con}_2(B)$, $(f^*)^{-1}(\{\nabla _A\})=\{\nabla _B\}$ and ${\rm Max}(A)\subseteq {\rm Spec}(A)$, then $f$ is admissible and Max--admissible.
\item\label{cardadm3} If ${\cal C}$ is congruence--modular, ${\rm Spec}(B)={\rm Con}_2(B)$, $(f^*)^{-1}(\{\nabla _A\})=\{\nabla _B\}$ and $\nabla _A$ is finitely generated, then $f$ is admissible and Max--admissible.
\item\label{cardadm4} If ${\cal C}$ is semi--degenerate and ${\rm Max}(B)={\rm Con}_2(B)$, then $f$ is Max--admissible.
\item\label{cardadm5} If ${\cal C}$ is semi--degenerate and congruence--modular and ${\rm Spec}(B)={\rm Max}(B)={\rm Con}_2(B)$, then $f$ is admissible and Max--admissible.\end{enumerate}\label{cardadm}\end{corollary}

\begin{proof} (\ref{cardadm1}) and (\ref{cardadm2}): By Lemma \ref{cardclase}, (\ref{cardclase4}), and Remark \ref{2max}.

\noindent (\ref{cardadm3}) By (\ref{cardadm2}) and Lemma \ref{folclor}.

\noindent (\ref{cardadm4}) By (\ref{cardadm1}) and Lemma \ref{gen01}.

\noindent (\ref{cardadm5}) By (\ref{cardadm3}), Remark \ref{rfolclor} and Lemma \ref{gen01}.\end{proof}

\begin{remark} Assume that ${\cal C}$ is congruence--modular and ${\rm Spec}(A)={\rm Max}(A)$, and let $M$ and $N$ be members of ${\cal C}$, $g:M\rightarrow A$ be a Max--admissible morphism and $h:A\rightarrow N$ an admissible morphism in ${\cal C}$. Then, clearly:\begin{itemize}
\item\label{specmax1} if ${\rm Max}(M)\subseteq {\rm Spec}(M)$, then $g$ is admissible;
\item\label{specmax2} if ${\rm Max}(N)\subseteq {\rm Spec}(N)$, then $h$ is Max--admissible;
\item\label{specmax3} thus, by Remark \ref{rfolclor}, if ${\cal C}$ is semi--degenerate, then $g$ is admissible and $h$ is Max--admissible.\end{itemize}\label{specmax}\end{remark}

\section{More Results on Admissibility and Max--admissibility}
\label{moreon(max)adm}

In this section, we prove that surjectivity implies admissibility and Max--admissibility, but the converse is not true, that the study of admissibility and Max--admissibility can be reduced to canonical embeddings, that admissibility and Max--admissibility are preserved by quotients, and several other results.

\begin{remark} If $f^*:{\rm Con}(B)\rightarrow {\rm Con}(A)$ is a bounded lattice isomorphism, then:\begin{itemize}
\item clearly, $f^*({\rm Max}(B))={\rm Max}(A)$; in particular, $f$ is Max--admissible;
\item by Remark \ref{primdistrib}, if ${\cal C}$ is congruence--distributive, or ${\cal C}$ is congruence--modular and the commutator in $A$ and $B$ equals the intersection of congruences, then $f^*({\rm Spec}(B))={\rm Spec}(A)$; in particular, $f$ is admissible.\end{itemize}\label{isomlatcongr}\end{remark}

\begin{proposition}\begin{enumerate}
\item\label{admmaxadm1} If ${\cal C}$ is congruence--modular, then any surjective morphism in ${\cal C}$ is admissible, but the converse is not true.
\item\label{admmaxadm2} Any surjective morphism is Max--admissible, but the converse is not true.\end{enumerate}\label{admmaxadm}\end{proposition}

\begin{proof} (\ref{admmaxadm1}) This is a result in \cite{gulo}, which uses \cite[Proposition 2.1, (1)]{agl} for the direct implication and provides counter--examples for the converse implication which also disprove the converse implication in (\ref{admmaxadm2}).

\noindent (\ref{admmaxadm2}) Assume that $f$ is surjective, and let $\mu \in {\rm Max}(B)$, so that $f^*(\mu )\in {\rm Con}(A)$, $\mu \neq \nabla _B$ and $(f^*)^{-1}(\{\nabla _A\})=\{\nabla _B\}$ by Lemma \ref{cardclase}, (\ref{cardclase5}), thus $f^*(\mu )\neq \nabla _A$. Let $\alpha \in {\rm Con}(A)$ such that $f^*(\mu )\subseteq \alpha $, thus, since $f$ is surjective, $\mu =f(f^*(\mu ))\subseteq f(\alpha )\in {\rm Con}(B)$. But $\mu \in {\rm Max}(B)$, hence $f(\alpha )=\mu $ or $f(\alpha )=\nabla _B$, so that $\alpha =f^*(f(\alpha ))=f^*(\mu )$ or $\alpha =\nabla _A$. Therefore $f^*(\mu )\in {\rm Max}(A)$, so $f$ is Max--admissible.

Example \ref{fiectip}, Lemma \ref{morflat}, Propositions \ref{admdistrib} and \ref{im01} provide us with infinitely many counter--examples for the converses of the implications from both (\ref{admmaxadm1}) and (\ref{admmaxadm2}).\end{proof}

\begin{proposition} If $f$ is surjective, then $f({\rm Con}(A))={\rm Con}(B)$ and $f({\rm Max}(A))={\rm Max}(B)$. The converse is not true.\label{moreonadm}\end{proposition}

\begin{proof} Assume that $f$ is surjective, so that $f({\rm Con}(A))\subseteq {\rm Con}(B)$ and, for all $\beta \in {\rm Con}(B)$, $f(f^*(\beta ))=\beta $, thus ${\rm Con}(B)=f(f^*({\rm Con}(B)))\subseteq f({\rm Con}(A))$ since $f^*({\rm Con}(B))\subseteq {\rm Con}(A)$. Thus ${\rm Con}(B)\subseteq f({\rm Con}(A))\subseteq {\rm Con}(B)$, hence $f({\rm Con}(A))={\rm Con}(B)$.

By Proposition \ref{admmaxadm}, (\ref{admmaxadm2}), $f$ is Max--admissible, so $f^*({\rm Max}(B))\subseteq {\rm Max}(A)$. Now let $\mu \in {\rm Max}(A)$, so that $f(\mu )\in {\rm Con}(B)\setminus \{\nabla _B\}$ by Lemma \ref{cardclase}, (\ref{cardclase5}). Let $\beta \in {\rm Con}(B)$ such that $f(\mu )\subseteq \beta $, so that $\mu \subseteq f^*(f(\mu ))\subseteq f^*(\beta )$, thus $f^*(\beta )=\mu $ or $f^*(\beta )=\nabla _A$, hence $\beta =f(f^*(\beta ))=f(\mu )$ or $\beta =\nabla _B$, again by Lemma \ref{cardclase}, (\ref{cardclase5}). Thus $f(\mu )\in {\rm Max}(B)$, hence $f({\rm Max}(A))\subseteq {\rm Max}(B)$. Therefore $f({\rm Max}(A))\subseteq {\rm Max}(B)=f(f^*({\rm Max}(B)))\subseteq f({\rm Max}(A))$, so $f({\rm Max}(A))={\rm Max}(B)$.

Let $i:{\cal D}\oplus {\cal L}_2\rightarrow V$ and $j:{\cal L}_2\oplus {\cal P}\rightarrow W$ be the canonical embeddings between the following bounded lattices, embeddings which are clearly not surjective:\vspace*{-5pt}

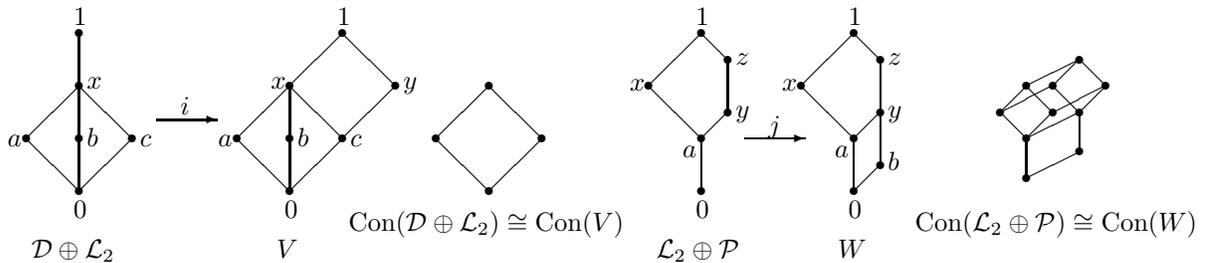
\begin{center}
\hspace*{-50pt}
\begin{tabular}{cccccccc}
\begin{picture}(40,100)(0,0)
\put(30,25){\line(0,1){40}}
\put(30,25){\line(1,1){20}}

\put(30,25){\line(-1,1){20}}
\put(30,65){\line(1,-1){20}}
\put(30,65){\line(-1,-1){20}}
\put(30,65){\line(0,1){20}}
\put(30,25){\circle*{3}}
\put(10,45){\circle*{3}}
\put(3,42){$a$}
\put(30,45){\circle*{3}}
\put(33,42){$b$}
\put(50,45){\circle*{3}}
\put(53,42){$c$}
\put(30,65){\circle*{3}}
\put(30,85){\circle*{3}}
\put(28,15){$0$}
\put(33,64){$x$}
\put(28,88){$1$}
\put(12,0){${\cal D}\oplus {\cal L}_2$}
\end{picture}
&\hspace*{-11pt}
\begin{picture}(40,100)(0,0)
\put(15,52){\vector(1,0){23}}
\put(24,54){$i$}
\end{picture}
&\hspace*{-20pt}
\begin{picture}(40,100)(0,0)
\put(30,25){\line(0,1){40}}
\put(30,25){\line(1,1){40}}
\put(30,25){\line(-1,1){20}}
\put(30,65){\line(1,-1){20}}
\put(50,85){\line(-1,-1){40}} 
\put(30,25){\circle*{3}}
\put(10,45){\circle*{3}}
\put(3,42){$a$}
\put(30,45){\circle*{3}}
\put(33,42){$b$}
\put(50,45){\circle*{3}}
\put(50,85){\circle*{3}}
\put(50,85){\line(1,-1){20}}
\put(70,65){\circle*{3}}
\put(53,42){$c$}
\put(30,65){\circle*{3}}
\put(28,15){$0$}
\put(23,64){$x$}
\put(73,64){$y$}
\put(48,88){$1$}
\put(25,0){$V$}
\end{picture}
&\hspace*{20pt}
\begin{picture}(40,80)(0,0)
\put(30,25){\line(1,1){20}}
\put(30,25){\line(-1,1){20}}
\put(30,65){\line(1,-1){20}}
\put(30,65){\line(-1,-1){20}}
\put(30,25){\circle*{3}}
\put(10,45){\circle*{3}}
\put(50,45){\circle*{3}}
\put(30,65){\circle*{3}}
\put(-23,10){${\rm Con}({\cal D}\oplus {\cal L}_2)\cong {\rm Con}(V)$}
\end{picture}
&\hspace*{25pt}
\begin{picture}(40,100)(0,0)
\put(30,45){\line(1,1){10}}
\put(30,45){\line(-1,1){20}}
\put(30,85){\line(1,-1){10}}

\put(30,85){\line(-1,-1){20}}
\put(40,55){\line(0,1){20}}
\put(30,45){\circle*{3}}
\put(30,25){\circle*{3}}
\put(30,25){\line(0,1){20}}
\put(10,65){\circle*{3}}
\put(3,63){$x$}
\put(40,55){\circle*{3}}
\put(43,53){$y$}
\put(30,85){\circle*{3}}
\put(40,75){\circle*{3}}
\put(43,73){$z$}
\put(28,15){$0$}
\put(23,38){$a$}
\put(28,88){$1$}
\put(13,0){${\cal L}_2\oplus {\cal P}$}
\end{picture}
&\hspace*{-24pt}
\begin{picture}(40,100)(0,0)
\put(15,45){\vector(1,0){23}}
\put(24,47){$j$}
\end{picture}
&\hspace*{-29pt}
\begin{picture}(40,100)(0,0)
\put(30,45){\line(1,1){10}}
\put(30,45){\line(-1,1){20}}
\put(30,85){\line(1,-1){10}}
\put(30,85){\line(-1,-1){20}}
\put(40,55){\line(0,1){20}}
\put(30,45){\circle*{3}}
\put(40,35){\circle*{3}}
\put(30,25){\circle*{3}}
\put(30,25){\line(0,1){20}}
\put(10,65){\circle*{3}}
\put(3,63){$x$}
\put(40,55){\circle*{3}}
\put(40,55){\line(0,-1){20}}
\put(43,53){$y$}
\put(43,33){$b$}
\put(30,85){\circle*{3}}
\put(40,75){\circle*{3}}
\put(30,25){\line(1,1){10}}
\put(43,73){$z$}
\put(28,15){$0$}
\put(23,38){$a$}
\put(28,88){$1$}
\put(24,0){$W$}
\end{picture}
&\hspace*{20pt}
\begin{picture}(40,100)(0,0)
\put(20,45){\line(1,1){10}}
\put(20,45){\line(-1,1){10}}
\put(20,65){\line(1,-1){10}}
\put(20,65){\line(-1,-1){10}}
\put(20,45){\line(0,-1){15}}
\put(20,30){\circle*{3}}
\put(20,45){\circle*{3}}
\put(10,55){\circle*{3}}
\put(30,55){\circle*{3}}
\put(20,65){\circle*{3}}
\put(40,55){\line(1,1){10}}
\put(40,55){\line(-1,1){10}}
\put(40,75){\line(1,-1){10}}
\put(40,75){\line(-1,-1){10}}
\put(40,55){\line(0,-1){15}}
\put(40,40){\circle*{3}}
\put(40,55){\circle*{3}}
\put(30,65){\circle*{3}}
\put(50,65){\circle*{3}}
\put(40,75){\circle*{3}}
\put(20,30){\line(2,1){20}}
\put(20,45){\line(2,1){20}}
\put(10,55){\line(2,1){20}}
\put(30,55){\line(2,1){20}}
\put(20,65){\line(2,1){20}}
\put(-22,10){${\rm Con}({\cal L}_2\oplus {\cal P})\cong {\rm Con}(W)$}\end{picture}
\end{tabular}\end{center}

Using Proposition \ref{ordsum} and Remark \ref{lant2}, it is easy to calculate that: ${\rm Con}(V)\cong {\rm Con}({\cal D}\oplus {\cal L}_2)\cong {\rm Con}({\cal D})\times {\rm Con}({\cal L}_2)\cong {\cal L}_2^2$, and $i^*:{\rm Con}(V)\rightarrow {\rm Con}({\cal D}\oplus {\cal L}_2)$ is a bounded lattice isomorphism, while ${\rm Con}(W)\cong {\rm Con}({\cal L}_2\oplus {\cal P})\cong {\rm Con}({\cal L}_2)\times {\rm Con}({\cal P})\cong {\cal L}_2\times {\rm Con}({\cal P})$, and $j^*:{\rm Con}(W)\rightarrow {\rm Con}({\cal L}_2\oplus {\cal P})$ is a bounded lattice isomorphism, hence $i^*({\rm Con}(V))={\rm Con}({\cal D}\oplus {\cal L}_2)$, $j^*({\rm Con}(W))={\rm Con}({\cal L}_2\oplus {\cal P})$ and, by Remark \ref{isomlatcongr}, $i^*({\rm Max}(V))={\rm Max}({\cal D}\oplus {\cal L}_2)$, $j^*({\rm Max}(W))={\rm Max}({\cal L}_2\oplus {\cal P})$ and we also have $i^*({\rm Spec}(V))={\rm Spec}({\cal D}\oplus {\cal L}_2)$ and $j^*({\rm Spec}(W))={\rm Spec}({\cal L}_2\oplus {\cal P})$.

Note, also, that the lattices of the congruences of ${\cal L}_2\oplus {\cal P}$ and $W$ are not Boolean algebras, and the prime congruences of ${\cal L}_2\oplus {\cal P}$ and $W$ do not coincide to their maximal congruences.\end{proof}

\begin{lemma} Let $C$ be a member of ${\cal C}$ and $g:B\rightarrow C$ be a morphism in ${\cal C}$.

\begin{enumerate}
\item\label{compadm1} If ${\cal C}$ is congruence--modular and $f$ and $g$ are admissible, then $g\circ f$ is admissible. 
\item\label{compadm2} If $f$ and $g$ are Max--admissible, then $g\circ f$ is Max--admissible.\end{enumerate}\label{compadm}\end{lemma}

\begin{proof} By the immediate fact that $(g\circ f)^*=f^*\circ g^*$. Note that (\ref{compadm1}) is a result in \cite{gulo}.\end{proof}

\begin{proposition} Let $i:f(A)\rightarrow B$ be the canonical embedding. Then:\begin{enumerate}
\item\label{admembed1} if ${\cal C}$ is congruence--modular, then: $f$ is admissible iff $i$ is admissible;
\item\label{admembed2} $f$ is Max--admissible iff $i$ is Max--admissible.\end{enumerate}\label{admembed}\end{proposition}

\begin{proof} (\ref{admembed1}) This is a result in \cite{gulo}.

\noindent (\ref{admembed2}) Let $g:A\rightarrow f(A)$, for all $a\in A$, $g(a)=f(a)$. Then $f=i\circ g$ and $g$ is a surjective morphism, thus $g$ is Max--admissible by Proposition \ref{admmaxadm}, (\ref{admmaxadm2}).\vspace*{-25pt}

\begin{center}\begin{picture}(120,53)(0,0)
\put(9,30){$A$}
\put(17,33){\vector (1,0){84}}
\put(15,28){\vector (3,-2){33}}	
\put(70,6){\vector (3,2){33}}
\put(53,38){$f$}
\put(25,11){$g$}
\put(89,11){$i$}
\put(102,30){$B$}
\put(48,0){$f(A)$}
\end{picture}\end{center}\vspace*{-5pt}

If $i$ is Max--admissible, then $f$ is Max--admissible by Lemma \ref{compadm}, (\ref{compadm2}).

Now assume that $f$ is Max--admissible and let $\mu \in {\rm Max}(B)$, so that $f^*(\mu )\in {\rm Max}(A)$, hence $g(f^*(\mu ))\in {\rm Max}(f(A))$ by Proposition \ref{moreonadm} and the fact that $g$ is surjective. $f=i\circ g$, hence $f^*=g^*\circ i^*$, thus $f^*(\mu )=g^*(i^*(\mu ))$, so $g(f^*(\mu ))=g(g^*(i^*(\mu )))=i^*(\mu )$, again by the surjectivity of $g$. Therefore $i^*(\mu )\in {\rm Max}(f(A))$, hence $i$ is Max--admissible.\end{proof}

\begin{remark} Clearly, if $M$ and $N$ are members of ${\cal C}$ and $g:M\rightarrow A$ and $h:B\rightarrow N$ are isomorphisms, then: $f$ is admissible, respectively Max--admissible, iff $f\circ g$ is admissible, respectively Max--admissible, iff $h\circ f$ is admissible, respectively Max--admissible.\label{compcuizom}\end{remark}

\begin{lemma} For any $\theta \in {\rm Con}(A)$, ${\rm Max}(A/\theta )=\{\mu /\theta \ |\ \mu \in {\rm Max}(A),\theta \subseteq \mu \}$.\label{speccat}\end{lemma}

\begin{proof} The mapping $\gamma \mapsto \gamma /\theta $ sets a bounded lattice isomorphism from $[\theta )$ to ${\rm Con}(A/\theta )$, thus a bijection from the set of the maximal elements of the lattice $[\theta )$, which, clearly, equals $[\theta )\cap {\rm Max}(A)$, to ${\rm Max}(A/\theta )$. Hence ${\rm Max}(A/\theta )=\{\mu /\theta \ |\ \mu \in [\theta )\cap {\rm Max}(A)\}=\{\mu /\theta \ |\ \mu \in {\rm Max}(A),\theta \subseteq \mu \}$.\end{proof}

Now let $\theta \in {\rm Con}(A)$ and let us define $f_{(\theta )}:A/\theta \rightarrow B/Cg_B(f(\theta ))$ by: for all $a\in A$, $f_{(\theta )}(a/\theta )=f(a)/Cg_B(f(\theta ))$. For any $a,b\in A$, if $a/\theta =b/\theta $, which means that $(a,b)\in \theta $, then $(f(a),f(b))\in f(\theta )\subseteq Cg_B(f(\theta ))$, thus $f(a)/Cg_B(f(\theta ))=f(b)/Cg_B(f(\theta ))$, so $f_{(\theta )}$ is well defined. Clearly, $f_{(\theta )}$ is a morphism and the following diagram is commutative:\vspace*{-25pt}

\begin{center}\begin{picture}(180,60)(0,0)
\put(31,35){$A$}
\put(24,5){$A/\theta $}
\put(143,35){$B$}
\put(122,5){$B/Cg_B(f(\theta ))$}
\put(23,22){$p_{\theta }$}
\put(34,33){\vector(0,-1){19}}
\put(150,22){$p_{Cg_B(f(\theta ))}$}
\put(147,33){\vector(0,-1){19}}
\put(86,42){$f$}
\put(40,38){\vector(1,0){100}}
\put(73,12){$f_{(\theta )}$}
\put(43,8){\vector(1,0){77}}\end{picture}\end{center}\vspace*{-10pt}

\begin{lemma}{\rm \cite{gulo}} Let $\theta \in {\rm Con}(A)$ and $\lambda \in [Cg_B(f(\theta )))$. Then $f^*(\lambda )\in [\theta )$ and $f_{(\theta )}^*(\lambda /Cg_B(f(\theta )))=f^*(\lambda )/\theta $.\label{f(theta)}\end{lemma}

\begin{proposition}\begin{enumerate}
\item\label{f(theta)adm1} $f$ is admissible iff, for any $\theta \in {\rm Con}(A)$, $f_{(\theta )}$ is admissible;
\item\label{f(theta)adm2} $f$ is Max--admissible iff, for any $\theta \in {\rm Con}(A)$, $f_{(\theta )}$ is Max--admissible.\end{enumerate}\label{f(theta)adm}\end{proposition}

\begin{proof} (\ref{f(theta)adm1}) This is a result in \cite{gulo}.

\noindent (\ref{f(theta)adm2}) For the converse implication, just take $\theta =\Delta _A$, so that $f(\theta )=f(\Delta _A)\subseteq \Delta _B$, thus $Cg_B(f(\theta ))=\Delta _B$, and so $p_{\theta }=p_{\Delta _A}$ and $p_{Cg_B(f(\theta ))}=p_{\Delta _B}$ are isomorphisms. Then, by Remark \ref{compcuizom}: $f(\Delta _A)$ is Max--admissible iff $f(\Delta _A)\circ p_{\Delta _A}=p_{\Delta _B}\circ f$ is Max--admissible iff $f$ is Max--admissible.

Now assume that $f$ is admissible, and let $\theta \in {\rm Con}(A)$ and $\mu \in {\rm Max}(B/Cg_B(f(\theta )))$, so that $\mu =\psi /Cg_B(f(\theta ))$ for some $\psi \in {\rm Max}(B)\cap [Cg_B(f(\theta )))$ by Lemma \ref{speccat}. Then, by Lemmas \ref{f(theta)} and \ref{speccat}, $f_{(\theta )}^*(\mu )=f_{(\theta )}^*(\psi /Cg_B(f(\theta )))=f^*(\psi )/\theta \in {\rm Max}(A/\theta )$ since $f^*(\psi )\in [\theta )\cap {\rm Max}(A)$, because $\theta \subseteq f^*(\psi )$ and $f$ is Max--admissible. Thus $f_{(\theta )}$ is Max--admissible.\end{proof}

\section{A Few Simple Applications to Subdirectly Irreducible Algebras}
\label{subdirirred}

In this section, we present a small set of applications of the above, that includes known properties, but which here are immediately derived from the previous results. Throughout this section, we shall assume that every algebra $M$ is non--trivial, so that $\Delta _M$ is a proper congruence of $M$.

\begin{theorem} {\rm \cite[Corollary 2, p. 140]{birkhoff}, \cite[Theorem 3, p. 13]{bal}, \cite[Theorem 8.27, p. 139]{blyth}} The algebra $A$ is subdirectly irreducible iff ${\rm Con}(A)\setminus \{\Delta _A\}$ has a minimum.\label{sdirired}\end{theorem}

For example, ${\cal L}_2$, ${\cal D}$, ${\cal P}$, the lattice $E$ in Example \ref{fiectip} and the lattice $N$ in Example \ref{latn} are subdirectly irreducible.

\begin{corollary}\begin{enumerate}
\item\label{deltasdired1} $A$ is subdirectly irreducible iff $\Delta _A$ is strictly meet--irreducible in the lattice ${\rm Con}(A)$.
\item\label{deltasdired0} If ${\cal C}$ is congruence--modular, the commutator in $A$ equals the intersection and $A$ is subdirectly irreducible, then $\Delta _A\in {\rm Spec}(A)$.
\item\label{deltasdired2} If ${\cal C}$ is congruence--distributive and $A$ is subdirectly irreducible, then $\Delta _A\in {\rm Spec}(A)$. 
\item\label{deltasdired4} If ${\cal C}$ is congruence--modular, the commutator in $A$ equals the intersection and ${\rm Con}(A)$ is finite, then: $A$ is subdirectly irreducible iff $\Delta _A\in {\rm Spec}(A)$ iff $\Delta _A$ is meet--irreducible.
\item\label{deltasdired3} If ${\cal C}$ is congruence--distributive and ${\rm Con}(A)$ is finite, then: $A$ is subdirectly irreducible iff $\Delta _A\in {\rm Spec}(A)$ iff $\Delta _A$ is meet--irreducible.\end{enumerate}\label{deltasdired}\end{corollary}

\begin{proof} (\ref{deltasdired1}) By Theorem \ref{sdirired}.

\noindent (\ref{deltasdired0}), (\ref{deltasdired2}) By (\ref{deltasdired1}) and Corollary \ref{mypartic}, (\ref{mypartic1}).

\noindent (\ref{deltasdired4}), (\ref{deltasdired3}) By (\ref{deltasdired1}) and Corollary \ref{mypartic}, (\ref{mypartic4}).\end{proof}

\begin{remark} Assume that ${\cal C}$ is congruence--modular and the commutator in $A$ equals the intersection, or that ${\cal C}$ is congruence--distributive. Then, by Corollary \ref{deltasdired}, (\ref{deltasdired0}) and (\ref{deltasdired2}), if $A$ is subdirectly irreducible and $|A|\geq 3$, then $\Delta _A\in {\rm Spec}(A)$ and $|A/\Delta _A|=|A|\geq 3$, thus $\Delta _A\notin {\rm Con}_2(A)$, hence ${\rm Spec}(A)\nsubseteq {\rm Con}_2(A)$. By Corollary \ref{mysuper}, it follows that no bounded distributive lattice of cardinality at least $3$ can be subdirectly irreducible.\end{remark}

We recall that an equational class ${\cal V}$ is said to be {\em congruence--extensible} iff, for any member $M$ of ${\cal V}$ and any subalgebra $S$ of $M$, any congruence of $S$ extends to a congruence of $M$, that is, for any $\sigma \in {\rm Con}(S)$, there exists a $\mu \in {\rm Con}(M)$ such that $\mu \cap S^2=\sigma $. For instance, the class of lattices is congruence--extensible (\cite{bal}, \cite{birkhoff}, \cite{blyth}, \cite{cwdw}, \cite{gratzer}, \cite{schmidt}).

\begin{corollary} Assume that $A$ is subdirectly irreducible, and let $S$ be a subalgebra of $A$, such that the canonical embedding of $S$ into $A$ is admissible. Assume, furthermore, that ${\cal C}$ is congruence--distributive, or that ${\cal C}$ is congruence--modular and the commutator in $A$ and $S$ equals the intersection. Then:\begin{enumerate}
\item\label{embedsdired1} if ${\rm Con}(A)$ and ${\rm Con}(S)$ are finite, then $S$ is subdirectly irreducible;
\item\label{embedsdired2} if $A$ is finite, then $S$ is subdirectly irreducible;
\item\label{embedsdired0} if ${\rm Con}(A)$ is finite and any congruence of $S$ extends to a congruence of $A$, then $S$ is subdirectly irreducible.
\item\label{embedsdired3} if ${\rm Con}(A)$ is finite and ${\cal C}$ is congruence--extensible, then $S$ is subdirectly irreducible.\end{enumerate}\label{embedsdired}\end{corollary}

\begin{proof} Let $i:S\rightarrow A$ be the canonical embedding. The fact that $i$ is admissible means that $i^*(\phi )=\phi \cap S^2\in {\rm Spec}(S)$ for all $\phi \in {\rm Spec}(A)$.

\noindent (\ref{embedsdired1}) By Corollary \ref{deltasdired}, (\ref{deltasdired3}), and the fact that $i^*(\Delta _A)=\Delta _S$.

\noindent (\ref{embedsdired2}) By (\ref{embedsdired1}).

\noindent (\ref{embedsdired0}) By (\ref{embedsdired1}) and the clear fact that, in this case, ${\rm Con}(S)$ is finite, as well. 

\noindent (\ref{embedsdired3}) By (\ref{embedsdired0}).\end{proof}

We have seen in Lemmas \ref{folclor} and \ref{specdistrib} some situations in which ${\rm Max}(A)\subseteq {\rm Spec}(A)$. Let us take a quick look at the converse inclusion.

\begin{remark} Clearly, $\Delta _A\in {\rm Max}(A)$ iff ${\rm Con}(A)=\{\Delta _A,\nabla _A\}$.\label{deltamax}\end{remark}

\begin{corollary} Assume, that ${\cal C}$ is congruence--distributive, or ${\cal C}$ is congruence--modular and the commutator in $A$ equals the intersection. If $A$ is subdirectly irreducible and ${\rm Con}(A)\supsetneq \{\Delta _A,\nabla _A\}$, then $\Delta _A\in {\rm Spec}(A)\setminus {\rm Max}(A)$, so ${\rm Spec}(A)\nsubseteq {\rm Max}(A)$.\end{corollary}

\begin{proof} By Corollary \ref{deltasdired}, (\ref{deltasdired2}), and Remark \ref{deltamax}.\end{proof}

\begin{example} An example in the class of bounded lattices for the situation in the previous corollary is the pentagon, which is subdirectly irreducible, but has ${\rm Con}({\cal P})\supsetneq \{\Delta _{\cal P},\nabla _{\cal P}\}$, and, as we have seen in Example \ref{fiectip}, it has ${\rm Spec}({\cal P})=\{\Delta _{\cal P}\}\cup {\rm Max}({\cal P})\nsubseteq {\rm Max}({\cal P})$.\label{expentagon}\end{example}

\begin{theorem}{\rm \cite[Theorem 3.3.1, p. 65]{schmidt}, \cite[Theorem 17, p. 81]{gratzer}} Any finite distributive lattice is isomorphic to the congruence lattice of some finite lattice.\label{conlatfin}\end{theorem}

\begin{example} Let $K$ be the following finite distributive lattice:
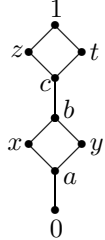
\vspace*{-14pt}

\begin{center}
\begin{picture}(40,80)(0,0)
\put(20,45){\line(1,1){10}}
\put(20,45){\line(-1,1){10}}
\put(20,65){\line(1,-1){10}}
\put(20,65){\line(-1,-1){10}}
\put(20,45){\line(0,-1){15}}
\put(18,68){$1$}
\put(20,30){\circle*{3}}
\put(14,40){$c$}
\put(3,53){$z$}
\put(33,52){$t$}
\put(20,45){\circle*{3}}
\put(10,55){\circle*{3}}
\put(30,55){\circle*{3}}
\put(20,65){\circle*{3}}
\put(20,10){\line(1,1){10}}
\put(20,10){\line(-1,1){10}}
\put(20,30){\line(1,-1){10}}
\put(20,30){\line(-1,-1){10}}
\put(20,10){\line(0,-1){15}}
\put(20,-5){\circle*{3}}
\put(18,-15){$0$}
\put(23,30){$b$}
\put(23,5){$a$}
\put(2,18){$x$}
\put(33,18){$y$}
\put(20,10){\circle*{3}}
\put(10,20){\circle*{3}}
\put(30,20){\circle*{3}}
\put(20,30){\circle*{3}}
\end{picture}\end{center}\vspace*{4pt}

By Theorem \ref{conlatfin}, it follows that there exists a finite lattice $L$ and a bounded lattice isomorphism $h:K\rightarrow {\rm Con}(L)$. Then, obviously, ${\rm Max}(L)=\{h(z),h(t)\}$, while ${\rm Spec}(L)=\{h(0)=\Delta _L,h(b),h(x),h(y),h(z),h(t)\}$, which is easily seen from Corollary \ref{mypartic}, (\ref{mypartic3}), thus $\Delta _L,h(b),h(x),h(y)\in {\rm Spec}(L)\setminus {\rm Max}(L)$. Note that, as shown by Corollary \ref{deltasdired}, (\ref{deltasdired3}), and Remark \ref{auconbool}, the lattice $L$ is subdirectly irreducible and can not be obtained through direct products and/or ordinal sums from modular lattices and relatively complemented lattices.\label{exlatcong}\end{example}

Let us generalize what we have observed in Examples \ref{expentagon} and \ref{exlatcong}:

\begin{remark}\begin{enumerate}
\item If $K$ is a finite distributive lattice, then, by Theorem \ref{conlatfin}, there exists a finite lattice $L$ such that ${\rm Con}(L)$ is isomorphic to ${\cal L}_2\oplus K$, thus $\Delta _L\in {\rm Spec}(L)$ and $L$ is subdirectly irreducible by Corollary \ref{deltasdired}, (\ref{deltasdired3}).
\item If $K$ and $M$ are finite distributive lattices and $c$ is the common element of $K$ and ${\cal L}_2$ in the ordinal sum $K\oplus {\cal L}_2\oplus M$, then, by Theorem \ref{conlatfin}, there exists a finite lattice $L$ and a bounded lattice isomorphism $h:K\oplus {\cal L}_2\oplus M\rightarrow {\rm Con}(L)$, thus, by Corollary \ref{deltasdired}, (\ref{deltasdired2}), $h(c)\in {\rm Spec}(L)$ and, clearly, if $M$ is non--trivial, then $h(c)\notin {\rm Max}(L)$, so ${\rm Spec}(L)\nsubseteq {\rm Max}(L)$.\end{enumerate}\end{remark}

\begin{corollary}\begin{itemize}
\item\label{latsdirired1} There are infinitely many subdirectly irreducible finite lattices.
\item\label{latsdirired2} There are infinitely many finite lattices $L$ with $\Delta _L\in {\rm Spec}(L)$.
\item\label{latsdirired3} There are infinitely many finite lattices $L$ with ${\rm Spec}(L)\nsubseteq {\rm Max}(L)$.\end{itemize}\label{latsdirired}\end{corollary}

\section*{Acknowledgements}

The author wishes to thank Professor Erhard Aichinger for his contribution to Proposition \ref{erhardgen}.

\end{document}